\documentclass[12pt,article]{amsart}
\usepackage[T1]{fontenc}
\usepackage[english]{babel}
\usepackage{enumerate}
\usepackage{mathtools}
\usepackage{esint}
\usepackage{setspace}
\usepackage{mathrsfs}
\usepackage{bbm}
\usepackage{geometry}
\usepackage{amssymb,amsthm,amsmath} 
\usepackage{indentfirst}
\usepackage{xcolor}
\usepackage{graphicx, caption}
\usepackage{pgfplots}
\usepackage{subfig}
\usepackage[countmax]{subfloat}
\renewcommand{\epsilon}{\varepsilon}

\makeatletter

\usepackage{url}									
\usepackage{hyperref}	
\usepackage{cleveref}

\usepackage{soul}
\usepackage{comment}

\newcommand{\numberset}{\mathbb}
\newcommand{\N}{\numberset{N}}
\newcommand{\R}{\numberset{R}}
\newcommand{\C}{\numberset{C}}
\newcommand{\Z}{\numberset{Z}}

\newcommand{\D}{\mathcal{D}}
\newcommand{\DAB}{\mathcal{D}_{AB}}

\DeclareMathOperator{\supp}{supp}
\DeclareMathOperator{\sgn}{sgn}

\newcommand{\green}{\textcolor{green}}

\newcommand{\blue}{\textcolor{blue}}

\numberwithin{equation}{section}
\newtheorem{proposition}{Proposition}[section]
\newtheorem{definition}{Definition}[section]
\newtheorem{lemma}{Lemma}[section]
\newtheorem{theorem}{Theorem}[section]
\newtheorem{corollary}{Corollary}[section]
\newtheorem{remark}{Remark}[section]

\newcommand{\frk}[1]{\ensuremath\mathfrak{#1}}

\theoremstyle{remark}

\theoremstyle{definition}

    \pgfplotsset{compat=1.18}
\begin{document}

\title[Generalized Strichartz estimates for the  Dirac equation]{Generalized Strichartz estimates for the massive Dirac equation with critical potentials}

\author{Federico Cacciafesta}
\address{Federico Cacciafesta: 
Dipartimento di Matematica, Universit\'a degli studi di Padova, Via Trieste, 63, 35131 Padova PD, Italy}
\email{cacciafe@math.unipd.it}

\author{Elena Danesi}
\address{Elena Danesi: DISMA, Politecnico di Torino,  Corso Duca degli Abruzzi, 24, 10129 Torino TO, Italy
}
\email{elena.danesi@polito.it}

\author{\'Eric S\'er\'e }
\address{\'Eric S\'er\'e: CEREMADE, UMR CNRS 7534, Universit\'e Paris-Dauphine, PSL Research University,
Pl. de Lattre de Tassigny,
75775 Paris Cedex 16, France}
\email{sere@ceremade.dauphine.fr}

\thanks{F.C. and E.D. acknowledge support by the Gruppo Nazionale per l'Analisi Matematica, la Probabilit\`{a} e le loro Applicazioni (GNAMPA)
}



\begin{abstract}
    In this paper we prove generalized Strichartz estimates for the massive Dirac equation in the case of two critical potential perturbations, namely the $2d$ Aharonov-Bohm magnetic potential and the $3d$ Coulomb potential. The proof makes use of the relativistic Hankel transform introduced in \cite{cacser}-\cite{cacfan} for the massless systems, and here adapted to the massive case: this allows for an explicit representation of the solutions, which reduces the analysis to the proof of  suitable estimates on the generalized eigenfunctions of the operators. To the best of our knowledge, these are the first dispersive estimates for the massive Dirac equation with critical potentials.
\end{abstract}
\maketitle

\section{Introduction}

The study of dispersive properties of the Dirac equation is a topic that has attracted a significant interest in recent years, given the relevance of the model in the applications (the Dirac equation is a fundamental equation in relativistic quantum mechanics, widely used to describe particles with spin $1/2$), and the intrinsic mathematical interest of the Dirac operator.

While dispersive estimates for ``small'' potentials are now well understood (see \cite{danfan}, \cite{bousdanfan}, \cite{erdgre}, \cite{erdgretop} and references therein), in the case of {\em critical potentials}, {\it i.e.} that exhibit the same homogeneity of the massless Dirac operator, and thus that are  critical with respect to the scaling of the equation, the picture is far from being completed and only some partial results are currently available, all of which in the {\em massless} case. We mention \cite{cacser}, \cite{cacserzha} and \cite{dan} for what concerns local smoothing and generalized Strichartz estimates for the massless Dirac-Coulomb equation, and \cite{cacfan}, \cite{cacyinzha} and \cite{cacdanzhayin1} for local smoothing, generalized Strichartz and Strichartz estimates for the massless Dirac equation in Aharonov-Bohm magnetic potential. 

The purpose of this paper is to complement the analysis developed in the papers above, by proving generalized Strichartz estimates for the Dirac equation with Aharonov-Bohm magnetic potential and then with Coulomb potential, in the {\em massive} case: the presence of a non-zero mass term yields additional difficulties, as indeed it destroys the scaling invariance of the operator and thus makes the analysis more delicate, as eventually high and low energies will exhibit a different behaviour.

\medskip

Let us start by recalling the definition of the free Dirac equation. In $2$ and $3$ dimensions, it is defined as follows:

\begin{equation}\label{diracfree}
\begin{cases}
\displaystyle
 i\partial_tu +\mathcal{D} u=0\,,\quad u(t,x):\mathbb{R}_t\times\mathbb{R}_x^n\rightarrow\mathbb{C}^{N}\\
u(0,x)=u_0(x)
\end{cases}
\end{equation}
where $N=2$ if $n=2$ and $N=4$ if $n=3$, the Dirac operator $\D$ is given by the formulas 
\begin{itemize}
\item $\mathcal{D}=-i\displaystyle\sum_{k=1}^3\alpha_k\partial_k+m\beta =-i(\alpha\cdot\nabla)+m\beta$ if $n=3$,
\item $\mathcal{D}=-i(\sigma_1\partial_x+\sigma_2\partial_y)+\sigma_3 m$ if $n=2$
\end{itemize}
where the constant $m$ (the mass) will be here strictly positive, and where
\begin{equation}
\beta= \left(\begin{array}{cc}Id_2 & 0 \\0 & -Id_2\end{array}\right),\quad
\alpha_k=\left(\begin{array}{cc}0 & \sigma_k \\\sigma_k & 0\end{array}\right),\quad k=1,2,3
\end{equation}
with $\sigma_k\; (k=1,2,3)$ the Pauli matrices
\begin{equation}\label{sigma}
\sigma_1=\left(\begin{array}{cc}0 & 1 \\1 & 0\end{array}\right),\quad
\sigma_2=\left(\begin{array}{cc}0 &-i \\i & 0\end{array}\right),\quad
\sigma_3=\left(\begin{array}{cc}1 & 0\\0 & -1\end{array}\right)\,.
\end{equation}

Before stating our main results, we find convenient to introduce some technical, preliminary tools that will be fundamental in the following.

\subsection{Preliminaries: partial wave decomposition}
The Dirac operator is a self-adjoint operator on $L^2(\R^n,\C^N)$ with domain $H^1(\R^n,\C^N)$, and it has purely continuous spectrum given by
$$
\sigma(\D)=(-\infty,-m]\cup[m,+\infty).
$$
For the basic properties of the free Dirac operator we refer to \cite{thaller}, to the survey \cite{estlewser} and references therein.

By making use of spherical coordinates it is possible to write 
\[
L^2(\R^n ; \C^N) \cong L^2 ( (0, \infty), r^{n-1}dr ) \otimes L^2 (\mathbb S^{n-1}; \C^N).
\]
Then, the following decompositions hold.
\begin{itemize}
\item {\bf Case $n=2$.} There is a unitary isomorphism
\[
L^2 (\R^2; \C^2 ) \cong \bigoplus_{k} L^2 ((0,\infty), r dr) \otimes \frk h^2_k
\]
given by
\begin{equation}
\label{decomp2}
\Psi (x) = \sum_{k \in \Z}  \Big [\psi^+_k (r) \,  \Xi^+_k (\theta) + \psi^-_k(r) \, \Xi^-_k (\theta)\Big ]
\end{equation}
where the spaces $\frk h^2_k$ are such that $L^2 (\mathbb S^1; \C^2) \cong \bigoplus_{k \in \Z }\frk h^2_k$, they are two-dimensional Hilbert spaces with basis $\big \{ \Xi^+_k(\theta), \Xi^-_k (\theta) \big \}$ given by
\begin{equation}
\label{circ_harm}
\Xi^+ _k (\theta) = \frac{1}{ \sqrt{2 \pi}} \begin{pmatrix}  e^{i k\theta} \\ 0 \end{pmatrix}, \quad 
\Xi^- _k (\theta) = \frac{1}{ \sqrt{2 \pi}} \begin{pmatrix} 0 \\ e^{i ( k + 1) \theta}  \end{pmatrix}.
\end{equation}
\item {\bf Case $n=3$.} There is a unitary isomorphism
\begin{equation*}
L^2(\R^3;\C^4) \cong \bigoplus_{k,m_k} L^2((0,\infty), r^2 dr) \otimes \frk h^3_{k,m_k}
\end{equation*}
given by
\begin{equation}
\label{decomp3}
\Psi (x)= \sum_{k \in \Z^*} \sum_{m_k} \psi^+_{k,m_k}(r) \Xi^+_{k,m_k}(\theta_1, \theta_2) + \psi^-_{k,m_k}(r) \Xi^-_{k,m_k} (\theta_1, \theta_2).
\end{equation}
where the spaces $\frk h^3_k$ are such that $L^2( \mathbb S^2; \C^4) \cong \bigoplus_{k \in \Z^*} \bigoplus_{m_k} \frk h^3_{k,m_k}$, $m_k = -\lvert k \rvert + \frac12,-\lvert k \rvert + \frac12 +1, \dots,  \lvert k \rvert - \frac12$.  they are two-dimensional Hilbert spaces with basis $\big \{ \Xi^+_{k,m_k}(\theta_1,\theta_2), \Xi^-_{k,m_k} (\theta_1,\theta_2) \big \}$ given by
\begin{equation}
\label{sph_harm}
\Xi^+_{k,m_k} (\theta_1,\theta_2)= \begin{pmatrix} i \Omega^{m_k}_{k} \\ 0_2 \end{pmatrix}, \quad \Xi^-_{k, m_k} = \begin{pmatrix} 0_2 \\ \Omega^{m_k}_{-k} \end{pmatrix}
\end{equation}
where
\[
\Omega_{k,m_k} = \frac1{\sqrt{\lvert 2k + 1 \rvert}} \begin{pmatrix} \sqrt{ \lvert k-m_k+ \tfrac12\rvert} Y^{m_k - \frac12}_{\lvert k + \frac12 \rvert - \frac12} \\ \sgn(-k) \sqrt{\lvert k+ m_k + \tfrac12 \rvert} Y^{m_k+\frac12}_{\lvert k + \frac12\rvert - \frac12} \end{pmatrix}.
\]
and where $Y^m_l(\theta_1, \theta_2)$ are the standard spherical harmonics.
\end{itemize}

The action of the operator $\D$ on each subspace $L^2(r^{n-1}dr)\otimes h_k$ is represented by the matrices:
\begin{itemize}
\item if $n=2$: 
\begin{equation}\label{2dfreerad}
d_k^{2dfree}=\left(\begin{array}{cc}-m &  -\frac{d}{dr}-\frac{k+1}{r} \\ \frac{d}{dr}-\frac{k}{r} & m\end{array}\right),\qquad k\in\Z
\end{equation}
\item if $n=3$:\begin{equation}\label{3dfreerad}
d_k^{3dfree}=\left(\begin{array}{cc}-m & -\frac{d}{dr}+\frac{k-1}{r} \\
\frac{d}{dr}+\frac{k+1}{r} & m\end{array}\right),\qquad k\in\Z^*.
\end{equation}
\end{itemize}
Notice that the formulas above only depend on the parameter $k$, both in dimensions $2$ and $3$; therefore, in order to provide a unified treatment, with a slight abuse of notations in what follows we will only keep the dependence on the parameter $k\in\mathcal{A}_n$ with $\mathcal{A}_2=\Z$ and $\mathcal{A}_3=\Z^*$.

The above decomposition provides a convenient separation of variables: any spinor $\Psi\in L^2(\R^n,C^N)$ can be written as
\begin{equation}\label{spherdec}
\Psi(x)=\sum_{k\in\mathcal{A}_n}\psi_k(r)\cdot \Xi_k(\theta)
\end{equation}
where $\psi_k(r)=\begin{pmatrix} \psi^1_k(r) \\
\psi_k^2(r)
\end{pmatrix}.$ 
This decomposition is often referred to as {\em partial wave decomposition}; in fact it is nothing but the analogue of the spherical harmonics decomposition adapted to the Dirac operator. We refer to \cite{thaller}, Section 4.6 for further details. 

In what follows, we shall define a function $\Psi$ to be {\em Dirac-radial} if every coefficient $\psi_k(r)$ in its decomposition \eqref{spherdec} is zero except for the ones corresponding to $k=0$ if $n=2$ and $k=\pm1$ if $n=3$. We thus introduce the following orthogonal projections on $L^{2}$ which will be useful:
\begin{itemize}
\item if $n=2$:
\begin{equation}\label{def-pro2}
  P_{rad}:
  L^{2}(\mathbb{R}^{2};\mathbb{C}^2)\to   
  L^{2}(rdr)\otimes h_{0}(\mathbb{S}^{1}),
  \qquad
  P_{\bot}=I-P_{rad},
\end{equation}
\item if $n=3$:
\begin{equation}\label{def-pro3}
  P_{rad}:
  L^{2}(\mathbb{R}^{3};\mathbb{C}^4)\to   
  \bigoplus_{k\in\{-1,1\}}L^{2}(r^{2}dr)\otimes h_{k}(\mathbb{S}^{2}),
  \qquad
  P_{\bot}=I-P_{rad}.
\end{equation}
\end{itemize}
Notice that for any initial datum $u_0\in Dom(\mathcal{D})$, the projections defined above allow for a decomposition of the flow of \eqref{diracfree} as follows:

 \begin{equation*}
e^{it\mathcal{D}} u_0=e^{it\mathcal{D}}P_{rad}u_0+e^{it\mathcal{D}}P_{\bot} u_0.
  \end{equation*}

\subsection{Main results}
We are now in position to state our main Theorems, that is, the generalized Strichartz estimates for the Dirac equations.


Our first result concerns the free Dirac equation in dimensions $2$ and $3$, that is system \eqref{diracfree}. 
\medskip

{\bf Functional setting and notations.} We shall use standard notations for Lebesgue and Sobolev spaces; if not specifically indicated, the norms will be intended on the whole space (i.e. $L^p_t=L^p_t(\R)$ and $L^q_x=L^q_x(\R^n)$), and we shall systematically omit the dimension on the target space. We shall denote with $L^p_tL^q_x=L^p(\mathbb{R}_t; L^q(\mathbb{R}^n_x))$ the mixed space-time Strichartz spaces. Using the polar coordinates $x=r\omega$, $r\geq0$, $\omega\in S^{n-1}$, and given a measurable function $F=F(t,x)$ we shall denote by
$$
\|F\|_{L^p_tL^q_{r^2dr}L^2_\omega}
:=
\left(\int_{\R}\left(\int_{0}^{+\infty}\left(\int_{ S^{n-1}}|F(t,r,\omega)|^2\,d\sigma\right)^{\frac q2}\,r^{n-1}dr\right)^{\frac pq}\,dt\right)^{\frac1p},
$$
being $d\sigma$ the surface measure on the sphere. We then have the following

\begin{theorem}\label{freetheorem}[Strichartz estimates in the free case].
Let $n=2,3$ and let $(p,q)$ be such that
\begin{itemize}
\item { { if $n=2$:}}
\begin{equation*}
2<p\leq \infty,\qquad 2\leq q<\infty,\qquad \frac1p+\frac1q<\frac12\:\: or\:\: (p,q)=(\infty,2),
\end{equation*}
\item { { if $n=3$:}}
\begin{equation*}
2\leq p\leq \infty,\qquad 2\leq q<\infty,\qquad \frac1p+\frac2q<1\:\: or\:\: (p,q)=(\infty,2).
\end{equation*}
\end{itemize}
Let $u_0\in  H^s(\mathbb{R}^n)$ and let us define
\begin{equation*}
    u_0^{low}=\chi_{[m,m+1]}(\vert\mathcal{D}\vert)u_0,\qquad u_0^{high}=u_0-u_0^{low}
\end{equation*}
where $\chi_A$ denotes the characteristic function of the set $A$.
Then there exists a constant $C>0$ such that the following Strichartz estimates hold
\begin{equation}\label{genstrichfreelow}
\|e^{it\mathcal{D}}u_0^{low}\|_{L^p_t L^q_{r^{n-1}dr}L^2_\theta}\leq C\|u_0^{low}\|_{ H^{\frac n2-\frac2p-\frac nq}},
\end{equation}
\begin{equation}\label{genstrichfreehigh}
\|e^{it\mathcal{D}}u_0^{high}\|_{L^p_t L^q_{r^{n-1}dr}L^2_\theta}\leq C\|u_0^{high}\|_{ H^{\frac n2-\frac1p-\frac nq}}.
\end{equation}
\end{theorem}

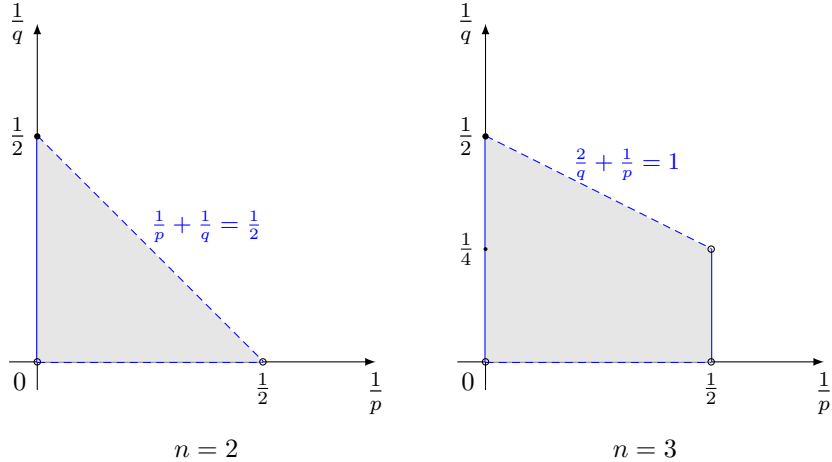
\begin{figure}
\centering
\subfloat[][ $n=2$]{%
\begin{tikzpicture}[scale=1.5, >=latex]
\draw[->] (-.25,0) -- (3,0) coordinate (x axis); \node[below] at (3,0) {$\frac1p$};
\draw[->] (0,-.25) -- (0,3) coordinate (y axis); \node[left] at (0,3) {$\frac1q$};
\draw[thick, densely dashed, blue] (0,2) -- (2, 0);
\draw[thick, densely dashed, blue] (0,0) -- (2, 0);
\draw[thick, blue] (0,2) -- (0,0);
\fill[gray!20] (0,0) -- (2,0) -- (0,2) -- cycle;
\draw (0,0) circle (0.8pt) node[below left] {\footnotesize $0$};
\draw (2,0) circle (0.8pt) node[below] {\footnotesize $\frac12$};
\draw (1.5,1.2) node [blue]{ \scriptsize$\frac1p + \frac 1q =\frac12$};
\fill (0,2) circle (0.8pt) node[left] {$\frac12$};
\end{tikzpicture}
}
\quad
\subfloat[][$n=3$]{
\begin{tikzpicture}[scale=1.5, >=latex]
\draw[->] (-.25,0) -- (3,0) coordinate (x axis); \node[below] at (3,0) {$\frac1p$};
\draw[->] (0,-.25) -- (0,3) coordinate (y axis); \node[left] at (0,3) {$\frac1q$};
\draw[thick, blue] (2,0) -- (2,1);
\draw[thick, densely dashed, blue] (0,0) -- (2, 0);
\draw[thick, blue] (0,2) -- (0,0);
\draw[thick, densely dashed, blue] (0,2) -- (2,1);
\fill[gray!20] (0,0) -- (2,0) -- (2,1) -- (0,2) -- cycle;
\draw (0,0) circle (0.8pt) node[below left] {\footnotesize $0$};
\draw (2,0) circle (0.8pt) node[below] {\footnotesize $\frac12$};
\draw (1.25,1.75) node [blue] {\scriptsize $\frac2q + \frac1p =1$};
\fill (0,1) circle (0.5pt) node [left] {\footnotesize $\frac14$};
\draw (2,1) circle (0.8pt);
\fill (0,2) circle (0.8pt) node[left] {$\frac12$};
\end{tikzpicture}
}
\caption{Generalized Strichartz estimates, free case}
\label{fig:gSe}
\end{figure}

\begin{remark}
Notice that estimates above imply in particular that for an initial datum which is Dirac-radial (as defined in the above) the standard Strichartz estimates hold, and in fact, for a wider range of exponents, as in the Klein-Gordon case, see \cite{ovcharov}. Also, with a relatively standard argument (see e.g. \cite{dan}), by making use of Sobolev embedding on the sphere, one could deduce from \eqref{genstrichfreelow}-\eqref{genstrichfreehigh} estimates on the complete $L^p_tL^q_x$ norms of the solution by imposing additional angular regularity on the initial condition. We omit the details.
\end{remark}

\begin{remark}\label{squaringrek}
Strichartz estimates for the free Dirac equation are usually deduced from the corresponding ones for the Klein-Gordon flow by making use of the well known  identity
\begin{equation}\label{squaringtrick}
(i\partial_t+\mathcal{D})(i\partial_t-\mathcal{D})=(\Delta-m^2-\partial_{tt}^2) I_4.
\end{equation}
Therefore, estimates \eqref{genstrichfreelow}-\eqref{genstrichfreehigh} could be obtained by the ones proved in \cite{ovcharov} for the propagator $e^{it\sqrt{-\Delta+m}}$. In fact, the solutions to the Klein-Gordon equation exhibit a ``wave-like'' behaviour for high energies and ``Sch\"odinger-like'' for low energies. This property is inherited by the solutions to the massive Dirac equation; therefore, in order to provide a precise result, we find convenient to treat separately the two different regimes, as done in \cite{ovcharov}. We prefer to provide the result and, most of all, the proof in the free case to begin with because it allows to explain in details the overall strategy and all the key estimates, so that dealing with the potential perturbations in the subsequent sections will be much easier and faster. We should point out indeed that our proof does not rely on the ``squaring trick'' \eqref{squaringtrick}: to the best of our knowledge, this is the first time that Strichartz estimates for the free flow are derived without the use of such identity. Also, we think that Theorem \ref{freetheorem} might provide a useful reference.
\end{remark}

Next, we consider the dynamics for the Dirac equation in Aharonov-Bohm magnetic potential, that is the two-dimensional system
\begin{equation}\label{DiracAB}
\begin{cases}
\displaystyle
 i\partial_tu+\mathcal{D}_{AB}u=0,\quad u(t,x):\mathbb{R}_t\times\mathbb{R}_x^2\rightarrow\mathbb{C}^{2}\\
u(0,x)=u_0(x)
\end{cases}
\end{equation}
where $\mathcal{D}_{AB}$ is given by
\begin{equation}\label{op:D}
\begin{split}
\DAB&=\begin{pmatrix}
    m & (i\partial_{1}+A^{1})-i(i \partial_{2}+A^{2})\\
    (i\partial_{1}+A^{1})+i(i \partial_{2}+A^{2}) &  -m
  \end{pmatrix}
\end{split}
\end{equation}
with the Aharonov-Bohm (AB) potential  $A(x)=(A^1(x),A^2(x))$ given by
\begin{equation}\label{AB}
A:\R^2\setminus\{(0,0)\}\to\R^2,
\quad
A(x)=\frac{\alpha}{|x|}\left(-\frac{x_2}{|x|},\frac{x_1}{|x|}\right),\, \quad
x=(x_1,x_2)
\end{equation}
and where the real constant $\alpha\in(0,1)$ is called the \emph{magnetic flux}. Note that 
for all nonzero $x\in \mathbb{R}^{2}$ we have
\begin{equation}\label{eq:transversal}
  \textstyle
A(x)\cdot\hat{x}=0,
\qquad
\hat{x}=\big(\frac{x_1}{|x|},\frac{x_2}{|x|}\big)\in\mathbb{S}^1.
\end{equation}

The (AB) potential \eqref{AB} is homogeneous of degree $-1$,  and thus it is a scaling critical perturbation for the massless Dirac operator. Dispersive estimates for solutions to \eqref{DiracAB} in the massless case are now well understood (see \cite{cacdanzhayin1}, also for a much wider overview of the model). On the other hand, to the best of our knowledge, nothing is known in the case $m>0$, and thus we here provide a first result.
\medskip

{\bf Functional setting.}
Letting
$
  \|u\|_{ H^{1}_{AB}}=
  \|\mathcal{D}_{AB}u\|_{L^{2}},
$
we define the $H^s_{AB}$ norms for $|s|\le1$ by interpolation and duality. This norm is not equivalent to the homogeneous Sobolev one, but it is possible to prove that $ \|u\|_{ H^{s}_{AB}}\leq  \|u\|_{  H^{s}}$ for any $s\in[0,1]$ 
(see e.g. Lemma 2.5 in \cite{cacfan}).

We then have the following

\begin{theorem}\label{ABtheorem}[Strichartz estimates for the Dirac-Aharonov Bohm equation].
  Let $\alpha\in(0, 1)$ and let $\mathcal{D}_{AB}^{dist}$ be the distinguished self-adjoint extension of $\mathcal{D}_{AB}$\footnote{The operator $\mathcal{D}_{AB}$ is not essentially self-adjoint, but a distinguished self-adjoint extension can be selected in a standard way. We postpone to subsection \ref{preliminariesAB} a small discussion of the issue.}. Assume $(p,q)$ satisfy
\begin{equation}\label{rangeAB}
p(\alpha)<p\leq \infty,\qquad 2\leq q<q(\alpha),\qquad\frac 2q + \frac{p(\alpha)}p \Big ( 1 - \frac 2{q(\alpha)} \Big ) <1 \:\: or\:\: (p,q)=(\infty,2)
\end{equation}
where
  \begin{equation}\label{q-alp}
  q(\alpha)=
  \begin{cases} 
    \frac2{\alpha}&
    \text{if}\quad \alpha\in(0,\frac12]\\
    \frac2{1-\alpha}&
    \text{if}\quad \alpha\in(\frac12,1)  \end{cases}\qquad {\rm and}\quad
    p(\alpha) =\frac{2q(\alpha) -4}{q(\alpha) -4}.
  \end{equation}
Let $u_0\in  H^s_{AB}(\mathbb{R}^2)$ and let us define
\begin{equation*}
    u_0^{low}=\chi_{[ m, (m+1)]}(\vert\mathcal{D}_{AB}^{dist}\vert)u_0,\qquad u_0^{high}=u_0-u_0^{low}
\end{equation*}
where $\chi_A$ denotes the characteristic function of the set $A$.
  Then there exists a constant $C>0$ such that the following Strichartz estimates hold
\begin{equation}\label{genstricdclow}
\|e^{it\mathcal{D}^{dist}_{A,m}}u_0^{low}\|_{L^p_t L^q_{rdr}L^2_\theta}\leq C\|u_0^{low}\|_{ H_{AB}^{1-\frac2p-\frac 2q}},
\end{equation}
\begin{equation}\label{genstricdchigh}
\|e^{it\mathcal{D}^{dist}_{A,m}}u_0^{high}\|_{L^p_t L^q_{rdr}L^2_\theta}\leq C\|u_0^{high}\|_{ H_{AB}^{1-\frac1p-\frac 2q}}.
\end{equation}
    If in particular $u_0$ is such that $P_{rad}u_0=0$, then $(p(\alpha),q(\alpha))$ in \eqref{rangeAB} becomes $(2,\infty)$.

\end{theorem}

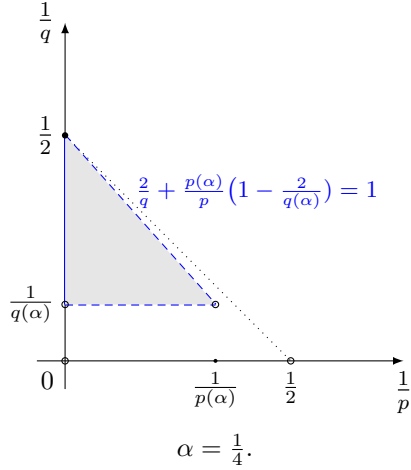
\begin{figure}
\centering
\subfloat[][$\alpha = \frac14 $.]{%
\begin{tikzpicture}[scale=1.5, >=latex]
\draw[->] (-.25,0) -- (3,0) coordinate (x axis); \node[below] at (3,0) {$\frac1p$};
\draw[->] (0,-.25) -- (0,3) coordinate (y axis); \node[left] at (0,3) {$\frac1q$};
\draw[dotted] (0,2) -- (2,0);
\draw[thick, densely dashed, blue] (0,2) -- (1.333, 0.5);
\draw[thick, densely dashed, blue] (0,0.5) -- ((1.333, 0.5);
\draw[thick, blue] (0,2) -- (0,0.5);
\fill[gray!20] (0,0.5) -- (1.333,0.5) -- (0,2) -- cycle;
\fill (1.333,0) circle (0.5pt) node[below] {\footnotesize $\frac1{p(\alpha)}$};
\draw (0,0) circle (0.8pt) node[below left] {\footnotesize $0$};
\draw (2,0) circle (0.8pt) node[below] {\footnotesize $\frac12$};
\draw (0,0.5) circle (0.8pt) node[left] {\footnotesize $\frac1{q(\alpha)}$};
\draw (1.333, 0.5) circle (0.8pt);
\draw (1.7,1.5) node [blue]{ \scriptsize$\frac2q + \frac{p(\alpha)}p \big ( 1 - \frac2{q(\alpha)}) =1$};
\fill (0,2) circle (0.8pt) node[left] {$\frac12$};
\end{tikzpicture}
}
\caption{Generalized Strichartz estimates, Dirac-Aharonov Bohm equation}
\label{fig:gSe}
\end{figure}

\begin{remark}
The restriction on the range of admissible Strichartz pairs in the case of the Aharonov-Bohm field provided by \eqref{q-alp} is ultimately due to the singularity of the generalized eigenfunctions in the origin: this feature seems in fact to provide a structural, unavoidable obstruction to the validity of the full set of estimates (see Proposition 6.6 in \cite{cacdanzhayin1}). In fact, in our proof we shall decompose the flow as  
\begin{equation*}
e^{it\mathcal{D}_{AB}} u_0=e^{it\mathcal{D}_{AB}}P_{rad} u_0+e^{it\mathcal{D}_{AB}}P_{\bot} u_0
  \end{equation*}
  and prove separate estimates for the two components: the component $e^{it\mathcal{D}_{AB}}P_{rad}u_0$ will be responsible for the restriction of the admissible range, as it is indeed the one that needs to account for the singularity of the domain. This singularity, as we shall see, is $\sim r^{-\alpha}$ when $\alpha \in(0,1/2]$ and $\sim r^{\alpha-1}$ when $\alpha \in (1/2,1)$. Therefore, $\alpha=1/2$ represents somehow a critical threshold value for the magnetic flux: with this choice, the singularity will be of order $ r^{-1/2}$, and we shall only be able to provide the (trivial) $L^\infty L^2$ bound for the flow.
  
Let us finally mention that it could be possible to recover the full set of admissible pairs as in the free case by introducing a suitable weight, and thus to obtain so-called ``weighted Strichartz estimates'', but we prefer to avoid further technicalities here.
\end{remark}

Finally, we consider the flow of the $3$d Dirac-repulsive Coulomb equation with positively projected initial datum, that is system
\begin{equation}\label{diraccoul}
\begin{cases}
\displaystyle
 i\partial_tu=\mathcal{D}_{DC} u\,,\quad u(t,x):\mathbb{R}_t\times\mathbb{R}_x^3\rightarrow\mathbb{C}^{4}\\
u(0,x)=u_0(x)
\end{cases}
\end{equation}
where $\mathcal{D}_{DC}$ is given by 
\begin{equation}\label{DCoperator}
\mathcal{D}_{DC}=\mathcal{D}-\frac{\nu}{|x|}\, \qquad \nu<0,
\end{equation}
where {\it we assume that $u_0$ is in the range of $\Pi_+:=\chi_{[0,+\infty)}(\mathcal{D}_{DC})$} (see \cite{mormul}). Note that this condition on $u_0$ implies that $u(t,\cdot)=\Pi_+u(t,\cdot)$ for all $t$, since $\Pi_+$ commutes with $\mathcal{D}_{DC}$.
For the sake of simplicity we shall focus on the $3d$ case; the $2d$ case could be dealt with with minor, technical differences. Let us stress that for this model the sign of the charge, which we are here taking to be  $\nu<0$ ({\em repulsive} Coulomb potential), plays a crucial role in the dynamics (see Remarks \ref{attractive-repulsive} and \ref{criticalregion}).

The Coulomb potential represents again a {\em scaling critical perturbation} for the massless Dirac operator, providing thus difficulties similar to the ones given by the AB potential. Anyway, the understanding of the dynamics in the massless case is at a much earlier stage: this is due on one hand to the fact that the generalized eigenfunctions exhibit a more involved expression (see below), and on the other hand to the fact that  scalar potentials are structurally more complicated then the magnetic ones, as the ``squaring trick'' \eqref{squaringtrick} does not allow to obtain a system of wave equations. As a further notable difference, the Dirac-Coulomb operator has a non-empty point spectrum, and the presence of eigenfunctions clearly represents an obstacle for dispersion. 

To our knowledge, the only available results for the massless dynamics are some local smoothing estimates (see \cite{cacser}) and the generalized Strichartz estimates proved in \cite{cacserzha}-\cite{dan}. Here, we provide a first result concerning the dynamics in the massive case.

{\bf Functional setting.}
We set
$
\|u\|_{{H}^1_{DC}}:=\||\mathcal{D}_{DC}|u\|_{L^2};
$
then, the $\|\cdot\|_{H^s_{DC}}$ for $|s|\leq 1$ will be obtained by interpolation and duality. We recall that, see \cite{fms21}, Corollary 1.8, the norm above is equivalent to the homogenous Sobolev one in dimension $n=3$ if $|\nu| < \frac{\sqrt3}2$.

We prove the following:
\begin{theorem}\label{DCtheorem}[Strichartz estimates for the positively projected 3d Dirac-Coulomb equation.]
Let $\nu\in(-\frac{\sqrt15}4,0)$ and let $\mathcal{D}_{DC}^{dist}$ be the distinguished self-adjoint extension of $\mathcal{D}_{DC}$\footnote{In this range of the charge, the Dirac-Coulomb operator admits a distinguished self-adjoint extension whose domain is contained in $H^{1/2}$. We shall briefly discuss this issue in subsection \ref{subsecDMprel}.} and assume $(p,q)$ satisfy
 \begin{equation}\label{rangedc}
2\leq p\leq \infty,\qquad 2\leq q<q(\nu),\qquad \frac1p+\frac2q<1\:\: or\:\: (p,q)=(\infty,2)
\end{equation}
where
  \begin{equation}\label{q-nu}
  q(\nu)= \frac 3{1 - \sqrt{1-\nu^2}}.
  \end{equation}
Let $u_0\in  H^s_{DC}$ and let us define
\begin{equation}\label{dcspeccond}
    u_0^{low}=\chi_{[m, m+1]}(\mathcal{D}_{DC}^{dist})u_0,\qquad u_0^{high}=\chi_{(m+1, +\infty)}(\mathcal{D}_{DC}^{dist})u_0
\end{equation}
where $\chi_A$ denotes the characteristic function of the set $A$.
  Then there exists a constant $C>0$ such that the following Strichartz estimates hold
\begin{equation}\label{genstrichdclow}
\|e^{it\mathcal{D}_{DC}^{dist}}u_0^{low}\|_{L^p_t L^q_{r^2dr}L^2_\theta}\leq C\|u_0^{low}\|_{ H_{DC}^{1-\frac2p-\frac 2q}},
\end{equation}
\begin{equation}\label{genstricdchigh}
\|e^{it\mathcal{D}_{DC}^{dist}}u_0^{high}\|_{L^p_t L^q_{r^2dr}L^2_\theta}\leq C\|u_0^{high}\|_{ H_{DC}^{1-\frac1p-\frac 2q}}.
\end{equation}
    If in particular $u_0$ is such that $P_{rad}u_0=0$, then we can take $q(\nu)=\infty$ in \eqref{rangedc}.
\end{theorem}

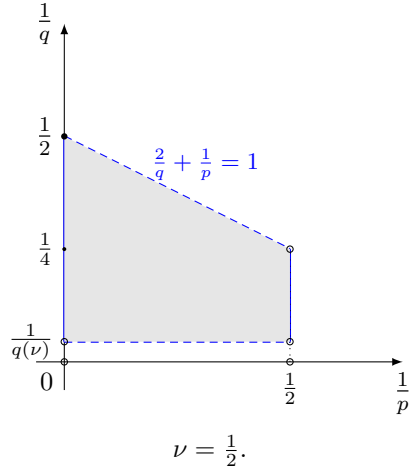
\begin{figure}
\centering
\subfloat[][$\nu=\frac12$.]{
\begin{tikzpicture}[scale=1.5, >=latex]
\draw[->] (-.25,0) -- (3,0) coordinate (x axis); \node[below] at (3,0) {$\frac1p$};
\draw[->] (0,-.25) -- (0,3) coordinate (y axis); \node[left] at (0,3) {$\frac1q$};
\draw[thick, blue] (2,0.179) -- (2,1);
\draw[thick, densely dashed, blue] (0,0.179) -- (2, 0.179);
\draw[thick, blue] (0,2) -- (0,0.179);
\draw[thick, densely dashed, blue] (0,2) -- (2,1);
\draw[dotted] (2,0) -- (2,1);
\fill[gray!20] (0,0.179) -- (2,0.179) -- (2,1) -- (0,2) -- cycle;
\draw (0,0) circle (0.8pt) node[below left] {\footnotesize $0$};
\draw (2,0) circle (0.8pt) node[below] {\footnotesize $\frac12$};
\draw (0,0.179) circle (0.8pt) node[left] {\footnotesize $\frac1{q(\nu)}$};
\draw (2, 0.179) circle (0.8pt);
\draw (1.25,1.75) node [blue] {\scriptsize $\frac2q + \frac1p =1$};
\fill (0,1) circle (0.5pt) node [left] {\footnotesize $\frac14$};
\draw (2,1) circle (0.8pt);
\fill (0,2) circle (0.8pt) node[left] {$\frac12$};
\end{tikzpicture}
}
\caption{Generalized Strichartz estimates, Dirac-Coulomb equation}
\label{fig:gSe}
\end{figure}

\begin{remark}
The condition on the charge $\nu\in(-\frac{\sqrt15}4,0)$ is purely technical, and it is a byproduct of our proof, and it is the same as for in the massless case obtained in \cite{cacserzha}.
\end{remark}

\begin{remark}\label{attractive-repulsive}
The above result presents some notable differences with respect to the Aharonov-Bohm one we should comment on. In this case indeed, in order to prove dispersive estimates we are forced to project onto the positive side of the spectrum; once the sign of the charge $\nu$ is fixed, positive and negative energies (in particular, for values close to $\pm m$) will indeed behave significantly differently. In particular, in our choice of the sign of $\nu<0$, the repulsive case, the point spectrum accumulates towards $-m$ (see formula \eqref{discspec}).  Thus, on the positive side of the spectrum,  far away from the eigenvalues, we shall be able to provide the desired pointwise bounds on the generalized eigenfunctions. On the other hand, for data with energies in the region $(-m-\epsilon,-m)$ with $\epsilon>0$, that are thus ``dangerously close'' to the point spectrum, the necessary bound on the generalized eigenfunctions does not hold. Therefore, our result does not provide estimates for the complete Dirac-Coulomb propagator, but only on its positively projected component. We refer to the Appendix, subsection \ref{confappendix} (see in particular Remark \ref{criticalregion}) for a more detailed technical motivation.
Let us mention the fact that this is perfectly consistent with the existing results: indeed, in the nonrelativistic limit $c\rightarrow \infty$, we expect the solutions to the massive Dirac equation to resemble the Schr\"odinger one, and dispersive estimates for the Schr\"odinger-Coulomb model have been proved in \cite{mizutani} only in the repulsive case.
\end{remark}

{\bf Plan of the paper.}
The next three sections will be devoted to the proofs of Theorems \ref{freetheorem}, \ref{ABtheorem} and \ref{DCtheorem} respectively. The strategy developed follows a now well-established path in the field (see \cite{cacserzha}-\cite{dan}). On the other hand, to the best of our knowledge, this is the very first time that it is adapted to the massive Dirac equation. Here, the presence of the mass term destroys the scaling of the equation, and thus a more refined analysis separating between high and low energies is required. The main technical step is given by Lemma \ref{localizedlemmafree} (in the free and AB case), and  the corresponding one \ref{localizedlemmaDC} (in the Dirac-Coulomb case), in which the structure of the generalized eigenfunctions is exploited in order to obtain the necessary decay on the solution. Let us mention that while Lemma \ref{localizedlemmafree} will be fairly easy to be proved, due to the fact that in the free and in the AB case the generalized eigenfunctions are essentially standard Bessel functions, Lemma \ref{localizedlemmaDC} will require additional technicalities, that is a precise pointwise estimates for the  eigenfunctions of the Dirac-Coulomb model in the massive case. This will be proved in the appendix. Then, the punchline developed in the previous papers will give the results.

\section{The free case: proof of Theorem \ref{freetheorem}}\label{freesec}

In this section we provide the proof of Theorem \ref{freetheorem}. As mentioned in the introduction, the strategy will be later adapted to deal with the case of potential perturbations; therefore, providing all the details here in this much simpler framework will allow to significantly simplify the presentation later on.

\subsection{Spectral theory.}\label{specsec}

In view of providing an explicit representation for the semigroup $e^{it\D}$, we need to write down the eigenfunctions of the operator $\D$. Due to decomposition \eqref{spherdec}, 
we are interested in solving the eigenvalue equation
\begin{equation}\label{eq:geneg}
d_{k} \Psi_k(E, r)=d_{k}\left(\begin{matrix} \psi_k^1(E, r)\\ \psi_k^2(E, r)\end{matrix}\right)=E  \left(\begin{matrix} \psi_k^1(E, r)\\ \psi_k^2(E, r)\end{matrix}\right)
\end{equation}
with $d_k$ given by \eqref{2dfreerad}-\eqref{3dfreerad}. After some simple manipulations, system \eqref{eq:geneg} takes the form:\\
$\bullet$ in $2d$:
\begin{equation}\label{sys-bessel2D}
\begin{cases}
  \frac{d^2}{dr^2} \psi^1_k+\frac1r  \frac{d}{dr} \psi^1_k+((E^2-m^2)-\frac{k^2}{r^2})\psi^1_k=0,\\
    \frac{d^2}{dr^2} \psi^2_k+\frac1r  \frac{d}{dr} \psi^2_k+((E^2-m^2)-\frac{(k+1)^2}{r^2})\psi^2_k=0,
  \end{cases}
\end{equation}
$\bullet$ in $3d$:
\begin{equation}\label{sys-bessel3D}
\begin{cases}
  \frac{d^2}{dr^2} \psi^1_k+\frac1r  \frac{d}{dr} \psi^1_k+((E^2-m^2)-\frac{k(k+1)}{r^2})\psi^1_k=0,\\
    \frac{d^2}{dr^2} \psi^2_k+\frac1r  \frac{d}{dr} \psi^2_k+((E^2-m^2)-\frac{k(k-1)}{r^2})\psi^2_k=0.
  \end{cases}
\end{equation}

These are systems of Bessel (resp. spherical Bessel) differential equations. The solutions for any $E\in\R$ to these systems are given by the following (here and in the following we shall neglect some absolute constants in order to simplify the presentation):

$\bullet$ in $2d$:
\begin{equation}\label{2dfreesol}
\Psi_k(E,r) =\left(\begin{matrix} \psi_k^1(E, r)\\ \psi_k^2(E, r)\end{matrix}\right)=\begin{pmatrix} F_k^{2dfree}({\bf p},r) \\
G_k^{2dfree}({\bf p},r)
\end{pmatrix}=\displaystyle  \begin{pmatrix} N^+(E) J_{|k|}({\bf p}r) \\
-i {\rm sgn}(E)N^-(E)  J_{|k|+{\rm sgn}(k)} ({\bf p}r)
\end{pmatrix},\: k\in\Z
\end{equation}
(with the convention that ${\rm sgn}(k)=1$ if $k=0$);

$\bullet$ in $3d$:

\begin{equation}\label{3dfreesol}
\Psi_k(E,r) =\left(\begin{matrix} \psi_k^1(E, r)\\ \psi_k^2(E, r)\end{matrix}\right)=\begin{pmatrix}F_k^{3dfree}({\bf p},r) \\
G_k^{3dfree}({\bf p},r)
\end{pmatrix}=\displaystyle \sqrt{\frac{\pi}{2r}} \begin{pmatrix} N^+(E) J_{|k|+\frac12{\rm sgn}(k)}({\bf p}r) \\
-i {\rm sgn}(E)N^-(E)  J_{|k|-\frac12{\rm sgn}(k)} ({\bf p}r)
\end{pmatrix},\: k\in\Z^*
\end{equation}
where we are denoting with ${\bf p}=\sqrt{E^2-m^2}$ and with
\begin{equation}\label{Enorm}
N^\pm(E)=\sqrt\frac{{|E|\pm m}}{2|E|}.
\end{equation}

\begin{remark}
Let us give a few comments. First of all, notice that the spinors $\Psi_k$ do not belong to the space $L^2(r^{n-1}dr)$, as indeed the Bessel functions $J_k(r)\sim r^{-1/2}$ for $r\rightarrow +\infty$ (they are only locally in $L^2$) and thus they do not belong to the domain of the Dirac operator; they are indeed generalized eigenfunctions.
Notice moreover that the order of Bessel functions in formulas above is always positive; this implies that the generalized eigenfunctions do not exhibit a singularity in the origin. This won't be the case in presence of critical potentials as we shall see, as in fact generalized eigenfunctions will be unbounded near $0$, and this will have a significant byproduct in the range of admissible Strichartz pairs.
Finally, let us comment on the normalization term $N^\pm(E)$, which is motivated by the structure of the equation. The eigenvalue equation 
$$
\mathcal{D}\Psi=E\Psi
$$can be indeed equivalently written as follows (let us write for brevity only the $3d$ version)
\begin{equation}\label{eigeneqfree}
\begin{cases}
(E-m)\Psi^{up}+i\nabla\cdot \sigma \Psi^{down}=0,\\
(E+m)\Psi^{down}+i\nabla\cdot \sigma \Psi^{up}=0,
\end{cases}
\end{equation}
where $\Psi^{up}$ and $\Psi^{down}$ denote respectively the two-spinors of the top two and bottom two components of the four-spinor $\Psi$. Therefore, these components are connected through the following relations:
\begin{equation}\label{normcond}
\Psi^{up}=\frac{-i \sigma\cdot \nabla}{E-m}\Psi^{down},\qquad\Psi^{down}=\frac{-i \sigma\cdot \nabla}{E+m}\Psi^{up}.
\end{equation}
For all the details on the topic we refer to \cite{lanlifrel} Chapter 3 and \cite{lanlifnonrel} Chapter 5.
\end{remark}

We now define the {\em relativistic Hankel transform}, which will represent a fundamental object in the sequel.

\begin{definition}
Let $n=2,3$ and let $f(r)=\begin{pmatrix} f_1(r) \\
f_2(r)
\end{pmatrix}\in L^2(\R^n,\C^N)$. For any $k\in\mathcal{A}_n$ we define the following operator
\begin{equation}\label{freehank}
\mathcal P_k f(E) \coloneqq \int_0^{+\infty} \psi_k(E,r)^T f(r)r^{n-1} dr
\end{equation}
where $\psi_k$ is given by \eqref{2dfreesol}-\eqref{3dfreesol}. 
\end{definition}

The above operator satisfies several crucial properties that we collect into the following
\begin{proposition}\label{freehankprop}
Let $n=2,3$. The operator 
$$\mathcal P_k \colon [L^2((0,+\infty), r^{n-1} dr)]^2 \rightarrow L^2 (\R \setminus [-m,m], E dE)$$ defined by \eqref{freehank} satisfies for any $k\in\mathcal{A}_n$ the following properties:
\begin{enumerate}
\item $\mathcal P_k$ is invertible and the inverse $\mathcal P^{-1}_k \colon L^2(\R \setminus [-m,m], E dE) \rightarrow [L^2((0,+\infty), r^{n-1} dr )]^2$  is given by the sequence 
\begin{eqnarray}\label{freeinv}
\nonumber
\mathcal P_k^{-1} \phi (r') &=&\int_{\R\backslash[-m,m]}\ \begin{pmatrix}
F_k^{free}(E,r) & F_k^{free}(E,r)\\
  -i G_k^{free}(E,r) & -i G_k^{free}(E,r)
\end{pmatrix} 
\begin{pmatrix}
    \phi^+(E) \\ \phi^-(E)
\end{pmatrix}EdE\\
&:=& \begin{pmatrix}
    \mathcal H_k^{pos,+} & \mathcal H_k^{neg,+}\\
    - i \mathcal H_{k+1}^{pos,-} & - i \mathcal H_{k+1}^{neg,-} 
\end{pmatrix} 
\begin{pmatrix}
    \phi^+(E) \\ \phi^-(E)
\end{pmatrix}
\end{eqnarray}

where $F_k^{free}$ and $G_k^{free}$ are given by \eqref{2dfreesol}-\eqref{3dfreesol} and where $\phi^\pm (E) = \phi (E) \chi_{(\pm m, \pm \infty)}(E)$.
\item $\mathcal{P}_k$ is an isometry.
\item $(\mathcal P_k \D f) (E)= E \mathcal P_kf(E)$.
\end{enumerate}
\end{proposition}

\begin{proof}

\begin{enumerate}
\item Notice that with a change of variable we have 
\begin{eqnarray}
\label{Hpos}
\nonumber
\mathcal H_k^{pos, \pm }\phi (r') & =&\displaystyle \pm \int_m^\infty N^\pm(E)\frac1{r'^{\frac{n-2}2}} J_k({\bf p}r) \phi(E) E dE \\
& =& \displaystyle \pm \int_0^{+\infty} \mathcal{N}^{\pm}({\bf p}) \frac1{r'^{\frac{n-2}2}} J_k({\bf p}r) \phi(\sqrt{{\bf p}^2 +m^2}) {\bf p} d{\bf p},
\end{eqnarray}
\begin{eqnarray}
\label{Hneg}
\nonumber
\mathcal H_k^{neg, \pm }\phi (r') & =& \displaystyle  \int_{-\infty}^{-m} N^\mp(E) \frac1{r'^{\frac{n-2}2}} J_k({\bf p}r) \phi(E) E dE \\
& =& \displaystyle\int_0^{+\infty} \mathcal{N}^{\mp}({\bf p}) \frac1{r'^{\frac{n-2}2}} J_k({\bf p}r) \phi(-\sqrt{{\bf p}^2 +m^2}) {\bf p} d{\bf p},
\end{eqnarray}
with
\begin{equation}\label{gpm}
\mathcal{N}^\pm ({\bf p}) = \frac {\sqrt{ \sqrt{{\bf p}^2 +m^2} \pm m}}{\sqrt{ 2 \sqrt{{\bf p}^2+m^2}}}.
\end{equation}

Let $f^{up}(r)=\begin{pmatrix} \phi(r) \\
0
\end{pmatrix}$, and 
let us show that $(\mathcal P_k^{-1} \mathcal P_k f^{up})(r') = f^{up}(r')$. We have by definition
 \[
    (\mathcal P_k^{-1} \mathcal P_k f^{up})(r') = (\mathcal H_k^{pos,+} (\mathcal P_k f^{up})^+) (r') + (\mathcal H_k^{neg,+} (\mathcal P_k f^{up})^- (r') = I_1 + I_2.
    \]
    We compute the term separately. 
    \[
    \begin{split}
        I_1 & = \int_0^\infty \mathcal{N}^+({\bf p}) \frac1{r'^{\frac{n-2}2}} J_k({\bf p}r') \Big ( \int_0^\infty \psi_k (\sqrt{{\bf p}^2 +m^2}, r)^T \phi (r) r^{n-1} dr \Big ) {\bf p} d{\bf p} \\
        & = \int_0^ \infty \mathcal{N}^+({\bf p}) J_k({\bf p}r') \Big ( \int_0^\infty \frac 1{(rr')^{\frac{n-2}2}} \Big ( \mathcal{N}^+({\bf p}) J_k({\bf p}r) \phi_1(r) -i \mathcal{N}^-({\bf p}) J_{k+1}({\bf p}r) \phi_2(r) \Big ) r^{n-1} dr \Big ) {\bf p} d{\bf p}.
    \end{split}
    \]
    Analogously, we get
    \[
    I_2 =  \int_0^ \infty \mathcal{N}^-({\bf p}) J_k({\bf p}r') \Big ( \int_0^\infty \frac 1{(rr')^{\frac{n-2}2}} \Big ( \mathcal{N}^-({\bf p}) J_k({\bf p}r) \phi_1(r) +i \mathcal{N}^+({\bf p}) J_{k+1}({\bf p}r) \phi_2(r) \Big ) r^{n-1} dr \Big ) {\bf p} d{\bf p}
    \]
    We observe that 
    \[
    (\mathcal{N}^+({\bf p}))^2 + (\mathcal{N}^-({\bf p}))^2 = \frac {\sqrt{{\bf p}^2 +m^2} + m}{2 \sqrt{{\bf p}^2 +m^2}} +  \frac {\sqrt{{\bf p}^2 +m^2} - m}{2 \sqrt{{\bf p}^2 +m^2}} = 1.
    \]
    Therefore, summing $I_1$ with $I_2$ we get
    \[
     (\mathcal P_k^{-1} \mathcal P_k \phi )(r') e_1  = \int_0^\infty \int_0^{+\infty} \frac1{(rr')^{\frac{n-2}2}} J_k({\bf p}r') J_k({\bf p}r) \phi_1 (r) r^{n-1} dr {\bf p} d{\bf p} = \phi_1 (r')
    \]
    where, for the last equality we used the classical (formal) orthogonality relation
    \[
    \int_0^\infty J_k({\bf p}r) J_k({\bf p}r') {\bf p} d{\bf p} = \frac{\delta(r-r')}{r}
    \]
    (see e.g. \cite{pondeleon}). 
   The computations for the lower component $f^{down}(r)=\begin{pmatrix} 0\\
\phi(r)
\end{pmatrix}$ is analogous and we omit it.

\item 
Let us assume $\mathcal P_k^{-1} = \mathcal P_k^*$. Then
    \[
    \lVert \mathcal P_k f \rVert_{L^2_{EdE}}^2 = \langle \mathcal P_k f, \mathcal P_k f \rangle_{L^2_{EdE}} = \langle \mathcal P_k^{-1} \mathcal P_k f , f \rangle_{L^2_{r^2 dr}} = \lVert f \rVert_{L^2_{r^2 dr}}
    \]
    where for the last equality we use $i),ii)$. We are led then to the study of the adjoint $\mathcal P_k^*$. Let $f \in L^2 ((0,\infty), r^2 dr)^2$ and $g \in L^2 (\R \setminus [-m,m], E dE)$ . Then
    \[
    \begin{split}
    \langle \mathcal P_k f,g \rangle_{L^2_{EdE}} & = \int_0^\infty (\mathcal H_k^{pos,+} \bar g^+ f_1 + i \mathcal H_k^{pos,-} \bar g^+ f_2 )r^2 dr + \int_0^\infty ( \mathcal H_{k+1}^{neg,+} \bar g^- f_1 + i \mathcal H_{k+1}^{neg,-} \bar g^- f_2 )r^2 dr \\
    & = \int_0^\infty \Big [ ( \overline{ \mathcal H_k^{pos,+} g^+ } + \overline{\mathcal H_{k+1}^{neg,+} g^-}) f_1 + ( \overline{-i \mathcal H_k^{pos,-} g^+ } + \overline{-i \mathcal H_{k+1}^{neg,-} g^-}) f_2 \Big ] r^2 dr \\
    & = \langle f, \mathcal P_k^{-1} g \rangle_{L^2_{r^2 dr}}.
    \end{split}
    \]
\item This is direct computation, and in fact is an immediate consequence of the definition of $\mathcal{P}_k$ and of the self-adjointness of the Dirac operator.
\end{enumerate}
\end{proof}

Proposition above, as we shall see in the next subsection, represents the stepping stone of our argument, as indeed allows for an explicit and versatile representation of the propagator $e^{it\D}$.

\subsection{Generalized Strichartz estimates (proof of Theorem \ref{freetheorem})}

We consider the Dirac propagator $e^{it\mathcal{D}}u_0$, that is the solution to problem
\begin{equation}\label{diracfree}
\begin{cases}
\displaystyle
 i\partial_tu+\mathcal{D} u=0\,,\quad u(t,x):\mathbb{R}_t\times\mathbb{R}_x^n\rightarrow\mathbb{C}^{N}\\
u(0,x)=u_0(x)
\end{cases}
\end{equation}
with initial condition $u_0\in L^2(\R^n)$. 
 Due to decomposition \eqref{spherdec} and Proposition \ref{freehankprop}, the solution to \eqref{diracfree} can be represented as
 \begin{equation*}
e^{it\mathcal{D}}u_0=\sum_{k\in\mathcal{A}_n}u_k(t,r)\cdot \Xi(\theta)= 
\sum_{k\in\mathcal{A}_n}e^{itd_k}u_{0,k}\cdot \Xi(\theta)=\sum_{k\in\mathcal{A}_n}\mathcal{P}_k^{-1}[e^{itE} (\mathcal{P}_k u_{0,k})(E)](r) \cdot \Xi(\theta)
 \end{equation*}
 where $u_{0,k}(r)$ are the radial components of the initial datum $u_0(x)$ in decomposition \eqref{spherdec}. For the $u_k(t,r)$ we have the following representation:
    \begin{equation}\label{radialsolrep}
    \begin{split}
    u_k (t,r) =\begin{pmatrix}
       u^1_k(t,r)\\
       u^2_k(t,r)
    \end{pmatrix}= e^{it d_k} u_{0,k} (r)  = \mathcal P_k^{-1} ( e^{it E} \mathcal P_k (u_{0,k}) (E)) (r) = \begin{pmatrix}
        \mathcal H_k ^{pos,+} & \mathcal H_k ^{neg,+}\\
        -i \mathcal H_{k+1} ^{pos,-} & -i \mathcal H_{k+1} ^{neg,-}
    \end{pmatrix} 
    \begin{pmatrix}
        e^{itE} (\mathcal P_k u_{0,k})^+\\
        e^{itE} (\mathcal P_k u_{0,k})^-
    \end{pmatrix}.
    \end{split}
    \end{equation}
Thanks to the orthogonality of the angular components $\Xi(\theta)$, when estimating the norms we will be able to write:
\begin{equation}\label{globalflow}
\|e^{it\mathcal{D}}u_0\|_{L^p_t L^q_{r^{n-1}dr}L^2_\theta}\leq
\left(\sum_{k\in\mathcal{A}_n}\|u_k(t,r)\|_{L^p_tL^q_{r^{n-1}dr}}^2\right)^{\frac12}.
\end{equation}

  Let us thus focus on the radial components of $u_k$ separately; \eqref{radialsolrep} yields the following representations

    \begin{equation}\label{freeu1}
    u^1_k(t,r) = \int_0^\infty  F_k^{free}({\bf p},r) \Big [ e^{it \sqrt{{\bf p}^2 +m^2}} (\mathcal P_k u_{0,k})^+(\sqrt{{\bf p}^2 +m^2}) + e^{-it \sqrt{{\bf p}^2 +m^2}} (\mathcal P_k u_{0,k})^-(- \sqrt{{\bf p}^2 +m^2}) \Big ] {\bf p} d{\bf p}
    \end{equation}
    and
   
     \begin{equation}\label{freeu2}
    u^2_k(t,r) = \int_0^\infty  G_k^{free}({\bf p},r) \Big [  e^{it \sqrt{{\bf p}^2 +m^2}} (\mathcal P_k u_{0,k})^+(\sqrt{{\bf p}^2 +m^2}) +  e^{-it \sqrt{{\bf p}^2 +m^2}} (\mathcal P_k u_{0,k})^-(- \sqrt{{\bf p}^2 +m^2}) \Big ] {\bf p} d{\bf p}
    \end{equation}
    where $F_k^{free}$, $G_k^{free}$ are given by either \eqref{2dfreesol} or \eqref{3dfreesol}.

 We thus need to provide estimates onto integrals in the form
    \begin{equation}\label{startingint}
    I^\pm_{k,free} = \int_0^\infty  H_k^{free}({\bf p},r) e^{\pm it \sqrt{{\bf p}^2 +m^2}} f(\pm\sqrt{{\bf p}^2 +m^2}) {\bf p} d{\bf p}
    \end{equation}
    with $f \in L^2 ((m,+\infty), EdE)$, $E= \sqrt{{\bf p}^2 +m^2}$ and where $H_k^{free}$ is either $F_k^{free}$ or $G_k^{free}$. 
    We shall focus on $I^+_{k,free}$, that is the {\em positive energy sector} as the terms $I^-_{k,free}$ can be dealt with with minor modifications\footnote{We anticipate the fact that this same remark will hold in the case of the Aharonov-Bohm magnetic potential, while for the Coulomb potential the relative sign of the charge and the energy will make a major difference. }.
We rely on a Paley-Littlewood type decomposition: let  $\phi = \chi_{(\frac 12, 1]}$, so that $\sum_{N \in 2^\Z} \phi \big ( \frac {\bf p} N \big ) =1$ for ${\bf p} \in (0,+\infty)$. Then, by the embedding $l^2 \hookrightarrow l^q$ for any $q \le + \infty$, we can estimate as follows 
    \begin{eqnarray}\label{estI}
   \nonumber     \lVert I^+_{k,free} \rVert_{L^p_t L^q_{r^{n-1} dr}}^2 & =& \Big \lVert \sum_{N \in 2^\Z} \int_0 ^\infty H_k^{free}({\bf p},r)  e^{it \sqrt{{\bf p}^2 +m^2}} f(\sqrt{{\bf p}^2 +m^2}) \phi \big ( \frac pN \big ) {\bf p} d{\bf p} \Big \rVert_{L^p_t L^q_{r^{n-1} dr}}^2  \\
        & \le& \sum_{R\in 2^\Z} \Big \lVert   \sum_{N \in 2^\Z} \int_0 ^\infty H_k^{free}({\bf p},r) e^{it \sqrt{{\bf p}^2 +m^2}} f(\sqrt{{\bf p}^2 +m^2}) \phi \big ( \frac pN \big ) {\bf p} d{\bf p} \Big \rVert_{L^p_t L^q_{r^{n-1} dr}[R,2R]}^2 \\
      \nonumber  & \le& \sum_{R\in 2^\Z} \bigg ( \sum_{N \in 2^\Z}  \Big \lVert   \int_0 ^\infty H_k^{free}({\bf p},r) e^{it \sqrt{{\bf p}^2 +m^2}} f(\sqrt{{\bf p}^2 +m^2}) \phi \big ( \frac pN \big ) {\bf p} d{\bf p} \Big \rVert_{L^p_t L^q_{r^{n-1} dr}[R,2R]} \bigg ) ^2.
    \end{eqnarray}
    We perform the change of variables: $y = \frac {\bf p} N$ and then $\rho = Nr$ and continue the chain of inequalities to get
\begin{equation*}
       \lVert I^+_{k,free} \rVert_{L^p_t L^q_{r^{n-1} dr}}^2  \le \sum_{R\in 2^\Z} \bigg ( \sum_{N \in 2^\Z}   {N^{2}} \Big \lVert   \int_0 ^\infty H_k^{free}(Ny,r)  e^{it \sqrt{(Ny)^2 +m^2}} f(\sqrt{(Ny)^2 +m^2}) \phi ( y) y dy \Big \rVert_{L^p_t L^q_{r^{n-1} dr}[R,2R]} \bigg ) ^2
       \end{equation*}
\begin{equation}\label{estII}
        \le  \sum_{R\in 2^\Z}  \bigg ( \sum_{N \in 2^\Z}   {N^{2}} {N^\frac{n-2}2 N^{- \frac nq}} \Big \lVert   \int_0 ^\infty H_k^{free}(y,\rho) e^{it \sqrt{(Ny)^2 +m^2}} f(\sqrt{(Ny)^2 +m^2}) \phi ( y) y dy \Big \rVert_{L^p_t L^q_{\rho^{n-1} d\rho}[NR,2NR]} \bigg ) ^2.
    \end{equation}
Now, the main step consists in proving Strichartz estimates on an interval for frequency localized solutions. The key result is thus the following

\begin{lemma}\label{localizedlemmafree}
  Let $n=2,3$, $h \in C_c^\infty$ with $\supp h \subset [\frac12, 1] = \colon I$ and let $\nu\geq0$.
  Then, for any $(p,q)\in(2,+\infty)\times(2+\infty)$ and $N\in 2^{\Z}$ the following estimate holds 
    \begin{equation}\label{locfreeest}
      \Big \lVert   \int_0 ^\infty \rho^{- \frac{n-2}2} J_{\nu}(\rho y) e^{it \sqrt{(Ny)^2 +m^2}} h(y) dy \Big \rVert_{L^p_t L^q_{\rho^{n-1} d\rho}[R,2R]}\leq C_N^{\frac 2p}  \|h\|_{L^{2}} \begin{cases}
          R^{\frac nq} & \quad \text{ if } R<1,\\
          R^{\frac 1q + \big (1- \frac 2q \big ) \beta(p)} & \quad \text{ if } R\ge1,
      \end{cases}
    \end{equation}
    with
    \begin{equation}\label{beta}
    \beta(p)=\begin{cases}
\frac1p-\frac{n-1}2\qquad {\rm if}\: p\in[2,4),\\
\frac1p-\frac{3n-4}6\qquad {\rm if}\: p\in[4,+\infty)
\end{cases}
\end{equation}
   and where the constant
    \begin{equation}\label{CN2}
    C_N\sim\begin{cases}
    N^{-\frac 12} \quad \text{if }\: N=2^n,\\
    N^{-1}  \quad \text{if} \: N=2^{-n}.
    \end{cases}
    \end{equation}
    is independent of $\nu$.
\end{lemma}

\begin{proof}(of Lemma \ref{localizedlemmafree})
\underline{ \emph{Step 1: preliminary estimate}} We start by proving the following. Let 
\[
F[h](t) \coloneqq \int_0^\infty e^{it \sqrt{(Ny)^2 + m^2} } h(y) dy, \quad N \in 2^\Z, \quad h \in C_c^\infty, \quad \supp h \subset I=[\frac12, 1],
\]
then
\begin{equation}
\label{HY}
\lVert F[h](\cdot) \rVert_{L^p} \le C_N^\frac 2p \lVert h \rVert_{L^{p'}} 
\end{equation}
for any $p \in [2,+\infty]$ and $p'$ such that $\frac 1p + \frac 1{p'} =1$. The proof follows from the Hausdorff-Young inequality. Using the change of variable $z=\sqrt{(Ny)^2 + m^2}$, we write 
\begin{equation}
    F[h](t) = \int_{\mathbb R} e^{it z } H(z) dz=\widehat{H}(-t)
    \end{equation}
    where $H(z)=\chi_{z\geq m}h\left(\frac{1}{N}\sqrt{z^2-m^2}\right)\frac{z}{N\sqrt{z^2-m^2}}$ and $\widehat{H}$ is its Fourier transform. Then we write
\[
     \int_{\mathbb R} \vert H(z)\vert^{p'} dz
     =\int_I \vert h(y)\vert^{p'} \left(\frac{\sqrt{(Ny)^2 + m^2}}{N^2 y}\right)^{p'-1}dy.
         \]
    We remark that on the interval $I$, 
         $\frac{\sqrt{(Ny)^2 + m^2}}{N^2 y}\leq \sqrt{N^{-2} +4 m^2 N^{-4}}$.
    Thus, $H\in L^{p'}(\mathbb R)$ with
    \begin{equation*}
    \Vert H\Vert_{L^{p'}(\mathbb R)}\leq (N^{-2} +4 m^2 N^{-4})^{\frac{1}{2p}}\, \Vert h\Vert_{L^{p'}(I)}\sim\begin{cases}
    N^{-1/p} \Vert h\Vert_{L^{p'}} \quad \text{if }\: N=2^n,\\
    (2mN^{-2})^{1/p} \Vert h\Vert_{L^{p'}} \quad \text{if} \: N=2^{-n}.
    \end{cases}
   \end{equation*}
    Combining this estimate with the Hausdorff-Young inequality $\lVert\widehat{H}\rVert_{L^p} \lesssim
    \lVert H \rVert_{L^{p'}}$, we get a bound of the form \eqref{HY} with $C_N$ satisfying \eqref{CN2}.

\emph{\underline{Step 2: proof of \eqref{locfreeest} for $R<1$.}} 
Let $p \in [2,+\infty)$ and $\Omega$ be an interval. Then, by the embedding $\dot H^{\frac 12 - \frac 1q} (\Omega) \hookrightarrow L^q(\Omega)$, that holds for every $q \in[2, +\infty)$ and interpolation, we have 
\[
\begin{split}
&\Big \lVert   \int_0 ^\infty \rho^{-\frac{n-2}2} J_{\nu}(\rho y)  e^{it \sqrt{(Ny)^2 +m^2}} h(y) dy \Big \rVert_{L^p_t L^q_{ d\rho}[R,2R]}   \\
& \lesssim \Big \lVert   \int_0 ^\infty \rho^{-\frac{n-2}2} J_{\nu}(\rho y)  e^{it \sqrt{(Ny)^2 +m^2}} h(y) dy \Big \rVert_{L^p_t \dot H^{\frac 12 - \frac 1q}_{ d\rho}[R,2R]} \\
&  \lesssim \Big \lVert   \int_0 ^\infty \rho^{-\frac{n-2}2} J_\nu(\rho y)  e^{it \sqrt{(Ny)^2 +m^2}} h(y) dy \Big \rVert_{L^p_t L^2_{ d\rho}[R,2R]}^{\frac 12 + \frac 1q} \times \Big \lVert   \int_0 ^\infty \rho^{-\frac{n-2}2} J_{\nu}(\rho y)  e^{it \sqrt{(Ny)^2 +m^2}} h(y) dy \Big \rVert_{L^p_t \dot H^1_{ d\rho}[R,2R]}^{\frac 12 - \frac 1q}.
\end{split}
\]

From Minkowski inequality and \eqref{HY} we have 
\[
\begin{split}
    & \Big \lVert   \int_0 ^\infty \rho^{-\frac{n-2}2} J_\nu(\rho y)  e^{it \sqrt{(Ny)^2 +m^2}} h(y) dy \Big \rVert_{L^p_t L^2_{ d\rho}[R,2R]} \le C_N^\frac 2p \Big \lVert \rho^{-\frac{n-2}2} J_\nu(\rho y) h(y) \Big \rVert _{L^2_{d\rho} [R,2R] L^{p'}_y} \\
    &  \le C_N^{\frac 2p}  \Big \lVert h(y) \lVert \rho^{-\frac{n-2}2} J_\nu(\rho) \rVert_{L^2_{d\rho} [yR, 2yR]} \Big \rVert_{L^{p'}_y} \lesssim C_N^{\frac 2p} R^{\frac12} \lVert h \rVert_{L^2(I)}
\end{split}
\]
where in the last inequality we have used Lemma \ref{lembessel}.
 Analogously, we have that 
\[
\begin{split}
   & \Big \lVert   \int_0 ^\infty \rho^{-\frac{n-2}2} J_\nu(\rho y)  e^{it \sqrt{(Ny)^2 +m^2}} h(y) dy \Big \rVert_{L^p_t \dot H^1_{ d\rho}[R,2R]} \lesssim C_N^{\frac 2p} \Big \lVert h(y) \lVert (\rho^{-\frac{n-2}2} J_\nu(\rho))' \rVert_{L^2_{d\rho} [yR, 2yR]} \Big \rVert_{L^{p'}_y} \\
   & \lesssim C_N^{\frac 2p } R^{\frac12-1} \lVert h \rVert_{L^2(I)}
\end{split}
\]
where in the last inequality we have used Lemma \ref{lembessel} again. 
Therefore, we conclude that 
\[
\Big \lVert   \int_0 ^\infty \rho^{-\frac{n-2}2} J_\nu(\rho y)  e^{it \sqrt{(Ny)^2 +m^2}} h(y) dy \Big \rVert_{L^p_t L^q_{ d\rho}[R,2R]} \lesssim C_N^{\frac 2p} R^{\frac 1q}\lVert h \rVert_{L^2(I)}.
\]
getting \eqref{locfreeest}.\\ 
\emph{\underline{Step 3: proof of \eqref{locfreeest} for $R\geq 1$.}} We aim to deduce \eqref{locfreeest} by interpolation between
\begin{equation}
\label{est_Rbig_infty}
    \Big \lVert   \int_0 ^\infty \rho^{-\frac{n-2}2} J_\nu(\rho y)  e^{it \sqrt{(Ny)^2 +m^2}} h(y) dy \Big \rVert_{L^p_t L^\infty_{ d\rho}[R,2R]}   \lesssim C_N^{\frac 2p } R^{\beta(p)} \lVert h \rVert_{L^2(I)},
\end{equation}
\begin{equation}
\label{est_Rbig_2}
    \Big \lVert   \int_0 ^\infty \rho^{-\frac{n-2}2} J_\nu(\rho y)  e^{it \sqrt{(Ny)^2 +m^2}} h(y) dy \Big \rVert_{L^p_t L^2_{\rho^{n-1} d\rho}[R,2R]}   \lesssim  C_N^{\frac 2p } R^{\frac 12} \lVert h \rVert_{L^2(I)}.
\end{equation}
In order to get \eqref{est_Rbig_infty}, we exploit the embedding $W^{1,p}(\Omega) \hookrightarrow L^\infty (\Omega)$, $\Omega $ an interval. Therefore, we have, by making use of inequality \eqref{HY}, 
\[
\begin{split}
 &\Big \lVert   \int_0 ^\infty \rho^{-\frac{n-2}2} J_\nu(\rho y)  e^{it \sqrt{(Ny)^2 +m^2}} h(y) dy \Big \rVert_{L^p_t L^\infty_{ d\rho}[R,2R]} \\ &\le \bigg ( \Big \lVert   \int_0 ^\infty \rho^{-\frac{n-2}2} J_\nu(\rho y)  e^{it \sqrt{(Ny)^2 +m^2}} h(y) dy \Big \rVert_{L^p_t L^p_{ d\rho}[R,2R]}   ^p + \Big \lVert   \int_0 ^\infty (\rho^{-\frac{n-2}2} J_\nu(\rho y))'  e^{it \sqrt{(Ny)^2 +m^2}} h(y) dy \Big \rVert_{L^p_t L^p_{ d\rho}[R,2R]} ^p \bigg )^\frac 1p\\
 & \lesssim C_N^{\frac 2p } \bigg ( \lVert \rho^{-\frac{n-2}2} J_\nu (\rho y) h(y) \rVert_{L^{p'}_y (I); L^p_{d\rho}[R,2R]}^ p  + \lVert (\rho^{-\frac{n-2}2} J_\nu (\rho y))' h(y) \rVert_{L^{p'}_y (I); L^p_{d\rho}[R,2R]}^ p\bigg )^\frac 1p\\
 & \lesssim C_N^{\frac 2p }  (R^{p \beta (p)})^\frac 1p\lVert h \rVert_{L^2(I)} = C_N^{\frac 2p } R^{\beta(p)} \lVert h \rVert_{L^2(I)}
\end{split} 
\]
where we have used Lemma \ref{lembessel}.
Estimate \eqref{est_Rbig_2} can be deduced similarly. Notice that we estimate here the $L^2$-norm with the measure $\rho^{n-1} d\rho$. Indeed
\[
\begin{split}
  & \Big \lVert   \int_0 ^\infty \rho^{-\frac{n-2}2} J_\nu(\rho y)  e^{it \sqrt{(Ny)^2 +m^2}} h(y) dy \Big \rVert_{L^p_t L^2_{\rho^{n-1} d\rho}[R,2R]}   \lesssim   C_N^{\frac 2p } R^\frac {n-1}2 \lVert \rho^{-\frac{n-2}2} J_\nu (\rho y) h(y) \rVert_{L^{p'}_y (I); L^2_{d\rho}[R,2R]}\\
   &  \lesssim   C_N^{\frac 2p } R^{\frac 12}  \lVert h \rVert_{L^2(I)}.
\end{split}
\]
Therefore, by interpolation
\[
 \Big \lVert   \int_0 ^\infty \rho^{-\frac{n-2}2} J_\nu(\rho y) e^{it \sqrt{(Ny)^2 +m^2}} h(y) dy \Big \rVert_{L^p_t L^q_{\rho^{n-1} d\rho}[R,2R]}\leq C_N^\frac 2p  R^{ \frac 1q + ( 1 - \frac 2q  ) \beta(p)}\lVert h \rVert_{L^2(I)}
\]
and this concludes the proof of the Lemma.
\end{proof}

We now plug \eqref{locfreeest} in \eqref{estII}: we thus choose either $\nu=\nu(k)\in \N$ if $n=2$ or $\nu=\nu(k)\in \N+\frac12$ if $n=3$. This yields (notice that we are here omitting the normalization terms $N^\pm(E)$ as they are cleary bounded by $1$ for $|E|\geq m$) \footnote{Notice that in fact, it might be possible to exploit the factor $N^-(E)$ to gain some additional regularity onto the estimates for low-frequencies for the ``negative spinor'' $f^{down}$, as indeed $N^-(E)\sim \frac{\bf p}{2m}$ for $\bf p\sim 0$, but we prefer not to insist on this as we are interested on estimates on the complete spinor. }

\begin{eqnarray}\label{finalstep}
\nonumber
\eqref{estII} &\leq &
\sum_{R\in 2^\Z}  \bigg ( \sum_{N \in 2^\Z}  C_N^{\frac2p}N^{\frac n2+1-\frac nq} Q_n(NR) \| f(\sqrt{(N{\bf p}))^2+m^2})\phi({\bf p})\|_{L^2_{{\bf p} d{\bf p}}} \bigg ) ^2
\\
&=&
\sum_{R\in 2^\Z}  \bigg ( \sum_{N \in 2^\Z}  C_N^{\frac2p}{N^{\frac n2 +1 - \frac nq-1}} Q_n(NR) \| f(\sqrt{{\bf p}^2+m^2})\phi(\frac{{\bf p}}N)\|_{L^2_{{\bf p} d{\bf p}}} \bigg ) ^2
\end{eqnarray}
with
\begin{equation*}
C_N\sim\begin{cases}
    N^{-\frac 12} \quad \text{if }\: N=2^n,\\
    N^{-1}  \quad \text{if} \: N=2^{-n}.
    \end{cases}
    \end{equation*}
    and
    \begin{equation*}
Q_n(NR)=\begin{cases}
          (NR)^{\frac nq } & \quad \text{ if } R<1,\\
          (NR)^{\frac 1q +  (1- \frac 2q ) \beta(p)} & \quad \text{ if } R\ge1
          \end{cases}
\end{equation*}
We now take $p,q$ such that 
\begin{equation}\label{freerange}
\frac nq >0,\qquad \frac 1q + \Big (1- \frac 2q \Big ) \beta(p)<0
\end{equation}
so that
$$
\sup_{R\in 2^\Z}\sum_{N\in 2^\Z}Q(NR)<+\infty,\qquad 
\sup_{N\in 2^\Z}\sum_{R\in 2^\Z}Q(NR)<+\infty.
$$
Recalling now that $f({\bf p})=(\mathcal{P}_ku_0)^+({\bf p})$, denoting with 
$A_{N,k}=C_N^{\frac2p}N^{\frac{n}2-\frac nq}\|f(\sqrt{{\bf p}^2+m^2})\phi(\frac{{\bf p}}N)\|_{L^2_{{\bf p}d{\bf p}}} $ and using Schur test Lemma we obtain
\begin{eqnarray*}
\left(\sum_{R\in 2^\Z} (\sum_{N\in 2^\Z} Q(NR)A_{N,k})^2\right)^{\frac12}&=&
\sup_{\| B_R\|_\ell^2\leq 1}\sum_{R\in2^\Z}\sum_{N\in 2^\Z}Q(NR)A_{N,k}B_R
\\
&\leq&C\left(\sum_{N\in 2^\Z}|A_{N,k}|^2\right)^{\frac12}.
\end{eqnarray*}
Putting all together and recalling \eqref{globalflow}, we eventually obtain
\begin{eqnarray}\label{freefinalstep}
\|e^{it\mathcal{D}}u_0^+\|^2_{L^p_t L^q_{r^2dr}L^2_\theta} &\leq&\sum_{k\in\mathcal{A}_n}
\sum_{R\in 2^\Z} (\sum_{N\in 2^\Z} Q(NR)A_{N,k})^2)^{\frac12}
\\
\nonumber
&\leq& \sum_{k\in\mathcal{A}_n}\sum_{N\in 2^\Z}
C_N^{\frac2p}{N^{\frac n2 - \frac nq}}
\| (\mathcal{P}_k u_0)^+(\sqrt{{\bf p}^2+m^2})\phi(\frac {\bf p}{N})\|_{L^2_{{\bf p} d{\bf p}}}.
\end{eqnarray}
The estimate for $e^{it\mathcal{D}}u_0^-$ is completely analogous, and we omit it.

As a consequence, recalling the fact that $\mathcal{P}_k$ is an isometry on $L^2$, we obtain the following generalized Strichartz estimates:
\begin{itemize}
\item {\bf ``Low energies''} (i.e. $N\sim2^{-n}$, that is ${\bf p}\sim m$). Given  $u_0\in L^2$, setting 
$$
u_{0}^{low}=\sum_{k\in\mathcal{A}_n}u_{0,k}^{low}=\sum_{k\in\mathcal{A}_n}\mathcal{P}_k^{-1}\chi_{\{{\bf p}\leq 1\}}\mathcal{P}_ku_0,
$$
estimate \eqref{freefinalstep} yields
\begin{equation}\label{lffreeest}
\|e^{it\mathcal{D}}u_{0}^{low}\|^2_{L^p_t L^q_{r^{n-1}dr}L^2_\theta}
\leq
\sum_{k\in\mathcal{A}_n}{N^{\frac n2 - \frac2p-\frac nq}}
\| (\mathcal{P}_k u_0)^+(\sqrt{{\bf p}^2+m^2})\phi(\frac {\bf p}{N})\|_{L^2_{{\bf p} d{\bf p}}}=
\|u_{0,k}^{low}\|_{H^{\frac n2-\frac2p-\frac nq}}
\end{equation}
\item {\bf ``High energies''}, (i.e. $N\sim2^{n}$, that is ${\bf p}\gg m$). Given  $u_0\in L^2$, setting for any $k\in\mathcal{A}_n$
$$
u_{0}^{high}=u_{0}-u_{0}^{low}
$$
with $u_{0,k}^{low}$ as above, estimate \eqref{freefinalstep} yields
\begin{equation}\label{hffreeest}
\|e^{it\mathcal{D}}u_{0,k}^{high}\|^2_{L^p_t L^q_{r^{n-1}dr}L^2_\theta}
\leq
\|u_{0,k}^{low}\|_{H^{\frac n2-\frac1p-\frac nq}}.
\end{equation}
\end{itemize}

Then, by interpolating with the standard estimate $\|e^{it\mathcal{D}}u_0\|_{L^\infty_tL^2_x}\leq \|u_0\|_{L^{2}}$ one gets the full set given by \eqref{genstrichfreelow}-\eqref{genstrichfreehigh}, and the proof is concluded.

    \section{Dirac-Aharonov Bohm system}\label{ABsec}

In this section we deal with the Dirac equation in Aharonov-Bohm magnetic potential, that is system \eqref{DiracAB}, and we prove Theorem \ref{ABtheorem}. The proof closely follows the one of the free case given in previous section. The computations are almost the same as in the free case, as indeed the generalized eigenfunctions of the operator $\mathcal{D}_{AB}$ still are Bessel functions. On the other hand, the key difference here is given by the fact that the presence of the magnetic potential ``produces a singularity'', meaning that (regardless of the choice of the self-adjoint extension) the domain of the operator $\mathcal{D}_{AB}$ will contain functions that are unbounded near zero. Equivalently, this means that (some of) the generalized eigenfunctions of the operator present a singularity in the origin (i.e. they will be Bessel functions of {\em negative} order). The presence of such a singularity  will reflect in a limitation on the range of admissible Strichartz pairs.

\subsection{Preliminaries}\label{preliminariesAB}

Let us start by briefly recalling some useful generalities about the self-adjoint extensions and spectral theory for the  the Dirac operator in Aharonov-Bohm magnetic potential $\DAB$, defined by \eqref{op:D}-\eqref{AB}.
By making use of spherical harmonics decomposition (the 2d version given by \eqref{decomp2}-\eqref{circ_harm}), 
the action of the operator $\DAB$ gets decomposed as follows
\begin{equation}\label{raddiropab}
  \DAB=
  \bigoplus_{k\in \mathbb{Z}}d_{k}^{AB}\otimes I_{2}
  \qquad
  d_{k}^{AB}:=
  \begin{pmatrix}
    -m & i(\frac d{dr}+\frac{k+1-\alpha}r)\\
    i(\frac d{dr}+\frac{\alpha-k}r) & m
  \end{pmatrix}, 
    \qquad
  I_2:=
  \begin{pmatrix}
    1 & 0 \\
  0 & 1 
  \end{pmatrix}.
\end{equation}

The self-adjoint extensions of the operator $\DAB$, initially defined on the minimal domain $C^\infty_0(\R^2)$ have been widely investigated, and are now well understood (see \cite{gerb}, \cite{dolestlos}, \cite{tam}, \cite{borcorten}. The operator $\DAB$ in fact is not essentially self adjoint (in particular, the operator $d_0^{AB}$ isn't); for any value of $\alpha\in(0,1)$, $\alpha\neq1/2$, it is possible to select a {\em distinguished} self-adjoint extension by requiring the condition that its domain embeds in $H^{1/2}$. Notice that this choice retrieves the ones in the physical literature for $\alpha\in(0,1/2)$, see \cite{gerb}, \cite{tam}. 
In the case $\alpha=1/2$ no extension that satisfies this condition exists: in this particular case, we shall (arbitrarily) select the one such that the singular term will be inherited entirely by the second component of the spinor. We should indeed stress the fact that for any value of $\alpha\in (0,1)$, every choice of the domain for the self-adjoint extension of $\DAB$ will include functions that exhibit a singularity in $r=0$: our choice of the distinguished extension is precisely the one that minimizes this singularity.

\medskip

Once the self-adjoint extension is fixed, we now need to write down the eigenfunctions of the operator $\DAB$. We are thus interested in solving the eigenvalue equation
\begin{equation}\label{eq:geneg}
d_{k}^{AB} \Psi_k(E, r)=d_{k}\left(\begin{matrix} \psi_k^1(E, r)\\ \psi_k^2(E, r)\end{matrix}\right)=E  \left(\begin{matrix} \psi_k^1(E, r)\\ \psi_k^2(E, r)\end{matrix}\right)
\end{equation}
with $d_k^{AB}$ given by \eqref{raddiropab}. This system is equivalent to the following
\begin{equation}\label{sys-bessel'}
\begin{cases}
  \frac{d^2}{dr^2} \psi^1_k+\frac1r  \frac{d}{dr} \psi^1_k+((E^2-m^2)-\frac{(k-\alpha)^2}{r^2})\psi^1_k=0,\\
    \frac{d^2}{dr^2} \psi^2_k+\frac1r  \frac{d}{dr} \psi^2_k+((E^2-m^2)-\frac{(k+1-\alpha)^2}{r^2})\psi^2_k=0
  \end{cases}
\end{equation}
which, again, is a system of Bessel equations (the AB magnetic potential has the effect of ``shifting'' the order of the Bessel functions). In our choice of the self-adjoint extension\footnote{Choosing a self-adjoint extension implies fixing boundary conditions for the functions of the domain, and thus the choice of the solutions to the eigenfunction equation.}, this system has the following solutions for $E\in(-\infty,-m]\cup[m,+\infty)$:

  \begin{equation}\label{psiAB}
    \psi_{k}^{AB}(E, r)= \begin{pmatrix} F_k^{AB}(E,r) \\
G_k^{AB}(E,r)
\end{pmatrix}=
    \frac{1}{\sqrt{2|E|}}
    \begin{pmatrix}
     N^+(E) J_{|k-\alpha|}({\bf p}r) \\
     - i {\rm sgn}(E)N^-(E)J_{|k-\alpha|+{\rm sgn}(k)}({\bf p}r)
    \end{pmatrix},\qquad {\rm if}\:k\in\Z^*
  \end{equation}
  and
  \begin{equation}\label{psiAB2}
    \psi_{0}^{AB}(E, r)= \begin{pmatrix}F_0^{AB}(E,r) \\
 G_0^{AB}(E,r)
\end{pmatrix}=
    \frac{1}{\sqrt{2|E|}}
    \begin{pmatrix}
     N^+(E)J_{-\alpha}({\bf p}r) \\
     - i {\rm sgn}(E)N^-(E)J_{1-\alpha}({\bf p}r)
    \end{pmatrix}
  \end{equation}
  where we recall ${\bf p}=\sqrt{E^2-m^2}$ and $N^\pm (E)$ are given by \eqref{Enorm}.
Notice that the above functions do not belong to the domain of $\DAB$, and are in fact generalized eigenfunctions. Notice also that the generalized eigenfunction $\psi_0^{AB}$, corresponding to the choice $k=0$, exhibits a singularity as $r^{-\alpha}$ in the origin.

\begin{remark}
It is important to remark that with our choice of the self-adjoint extension, the point spectrum is empty. For (almost) all the other choices, an eigenvalue in the gap $(-m,m)$ appears. We refer to \cite{gerb} for the details and the explicit computations.
\end{remark}

We are again in position to define the relativistic Hankel transform:
\begin{definition}
Let $f(r)=\begin{pmatrix} f_1(r) \\
f_2(r)
\end{pmatrix}\in L^2(\R^2,\C^2)$. For any $k\in\Z$ we define the following operator
\begin{equation}\label{ABhank}
\mathcal P_k^{AB} f(E) \coloneqq \int_0^{+\infty} \psi^{AB}_k(E,r)^T f(r)r dr
\end{equation}
where $\psi^{AB}_k$ is given by \eqref{psiAB}. 
\end{definition}

Then the following Proposition holds.
\begin{proposition}\label{ABankprop}
Let $f\in Dom(d_k^{AB})$. The operator 
$$\mathcal P_k^{AB} \colon [L^2((0,+\infty), r dr)]^2 \rightarrow L^2 (\R \setminus [-m,m], E dE)$$ defined by \eqref{ABhank} satisfies for any $k\in\Z$ the following properties:
\begin{enumerate}
\item $\mathcal P_k^{AB}$ is invertible and the inverse $\mathcal (P_k^{AB})^{-1} \colon L^2(\R \setminus [-m,m], E dE) \rightarrow [L^2((0,+\infty), r dr )]^2$  is given by the sequence 
\begin{eqnarray}\label{abinv}
\nonumber
\mathcal (P_k^{AB})^{-1}  \phi (r') 
&=&\int_{\R\backslash [-m,m]} \begin{pmatrix}
F_k^{AB}(E,r) & F_k^{AB}(E,r)\\
  -i G_k^{AB}(E,r) & -iG_k^{AB}(E,r)
\end{pmatrix} 
\begin{pmatrix}
    \phi^+(E) \\ \phi^-(E)
\end{pmatrix}EdE\\
\end{eqnarray}
where $F_k^{AB}$ and $G_k^{AB}$ are given by \eqref{psiAB}-\eqref{psiAB2}
and where $\phi^\pm (E) = \phi (E) \chi_{(\pm m, \pm \infty)}(E)$.

\item $\mathcal{P}_k^{AB}$ is an isometry.
\item $(\mathcal P_k^{AB} d_k^{AB} f) (E)= E \mathcal P_k^{AB}f(E)$.
\end{enumerate}
\end{proposition}

\begin{proof}
The proof is identical to the one of the free case, i.e. of Proposition \ref{freehankprop}, as indeed the generalized eigenfunctions present the same structure and all the properties of the standard Hankel transform are well known as long as the index of the Bessel function is larger than $-1/2$.

\end{proof}

\subsection{Generalized Strichartz estimates (proof of Theorem \ref{ABtheorem})}
The proof closely follows the one for the free case, and thus we just sketch it. In fact, the only difference one has to take into account is the index of Bessel functions in the generalized eigenfunction in the case $k=0$, i.e. for {\em radial functions} which, as seen in \eqref{psiAB}, is negative. This will imply a restriction on the range of admissible Strichartz pairs.

We consider the Cauchy problem
 \begin{equation}
        \label{eq:rad_AB}
        \begin{cases}
        i \partial_t u + \mathcal{D}_{AB} u = 0,\qquad  u \colon \R_t \times \R^2_x \rightarrow \C^2,\\
        u(0,x) = u_{0}.
        \end{cases}
    \end{equation}
    Relying on decomposition \eqref{spherdec}, on the operator $\mathcal{P}_k^{AB}$ as defined in \eqref{ABhank} and on Proposition \ref{ABankprop}, we can represent the solution as 
    \begin{equation*}
e^{it\mathcal{D}_{AB}}u_0=\sum_{k\in\Z}u_k(t,r)\cdot \Xi_k(\theta)=
\sum_{k\in\Z}e^{itd_k^{AB}}u_{0,k}\cdot \Xi_k(\theta)=\sum_{k\in\Z}(\mathcal{P}_k^{AB})^{-1}[e^{itE} (\mathcal{P}_k^{AB} u_{0,k})(E)](r) \cdot \Xi_k(\theta)
 \end{equation*}
 where $d_k^{AB}$ is given by \eqref{raddiropab} and the initial condition is $u_0(r)=\begin{pmatrix}
       u_0^1(r)\\
       u_0^2(r)
    \end{pmatrix}\in L^2((r^{n-1}dr)^2)$. 
This allows for an explicit representation of the components of the vector $ u_k (t,r) =\begin{pmatrix}
       u^1_k(t,r)\\
       u^2_k(t,r)
    \end{pmatrix}$
   in terms of the ones of the generalized eigenfunctions $\psi_k^{AB}$, analogous to \eqref{freeu1}-\eqref{freeu2}.

Without any modification from the free case, we get to the equivalent of estimate \eqref{estII}, that here reads as follows
\begin{equation*}
       \lVert I^+_{k,AB} \rVert_{L^p_t L^q_{r dr}}^2  \le 
       \sum_{R\in 2^\Z}  \bigg ( \sum_{N \in 2^\Z}   {N^{2-\frac2q}}  \Big \lVert   \int_0 ^\infty H_k^{AB}(y,\rho) e^{it \sqrt{(Ny)^2 +m^2}} f(\sqrt{(Ny)^2 +m^2}) \phi ( y) y dy \Big \rVert_{L^p_t L^q_{\rho d\rho}[NR,2NR]} \bigg ) ^2
       \end{equation*}
 where $f \in L^2 ((m,+\infty), EdE)$, $E= \sqrt{{\bf p}^2 +m^2}$ and $H_k^{AB}$ is either $F_k^{AB}$ or $G_k^{AB}$ as given by \eqref{psiAB}.
 
It is now convenient to deal separately with the cases $k=0$ (the ``radial part'') and $k\neq0$; to do this, we resort on the projectors $P_{rad}$ and $P_\bot$ as defined in \eqref{def-pro2}.
Recalling \eqref{spherdec}, for any initial datum $u_0\in D(\mathcal{D}_{AB})$ we can thus decompose the flow as follows:
 \begin{equation*}
e^{it\mathcal{D}_{AB}} u_0=e^{it\mathcal{D}_{AB}}P_{rad} u_0+e^{it\mathcal{D}_{AB}}P_{\bot} u_0
  \end{equation*}
  with 
  \begin{equation*}
  P_{rad} u_0=
 u_{0,0}(r)\cdot \Xi(\theta)
\end{equation*}
and 
  \begin{equation*}
  P_\bot u_0=  \sum_{k\in \mathbb{Z}\setminus\{0\}}
  u_{0,k}(r)\cdot \Xi_k(\theta).
\end{equation*}
\\\\
$\bullet$ {\bf Strichartz estimates for $e^{it\mathcal{D}_{AB}}P_{\bot} u_0$.} Here the proof follows exactly the lines of the one in the free case. The key Lemma \ref{localizedlemmafree} holds, with the choice of either $\nu=|k-\alpha|$ or $\nu=|k-\alpha|+{\rm sgn}(k)$, which are always strictly positive for $k\neq0$. Thus the whole proof can be retraced and the full set of Strichartz estimates are obtained. We omit the details.
\\\\
$\bullet$ {\bf Strichartz estimates for $e^{it\mathcal{D}_{AB}}P_{0} u_0$.} In this case the singularity of the domain (i.e. of the generalized eigenfunction) comes into play, as indeed the generalized eigenfunction takes the form $\psi_{0}^{AB}(E, r)=
    \frac{1}{\sqrt{2|E|}}
    \begin{pmatrix}
     \sqrt{|E|+m} J_{-\alpha}({\bf p}r) \\
     - i {\rm sgn}(E)\sqrt{|E|-m}J_{1-\alpha}({\bf p}r)
    \end{pmatrix}$.
    
    Its effect is easily understood: we need to modify the result of Lemma \ref{localizedlemmafree} as follows:
\begin{lemma}\label{localizedlemmaAB}
 Let $h \in C_c^\infty$ with $\supp h \subset [\frac12, 1] = \colon I$.
  Then, for any $\alpha \in (0, \frac12]$, any $(p,q)\in(2,+\infty)\times(2+\infty)$ and $N\in 2^{\Z}$ the following estimate holds 
    \begin{equation}\label{locABest}
     \Big \lVert   \int_0 ^\infty  J_{-\alpha}(\rho y) e^{it \sqrt{(Ny)^2 +m^2}} h(y) dy \Big \rVert_{L^p_t L^q_{\rho d\rho}[R,2R]}\leq C_{\alpha,N}^{\frac 2p}  \|h\|_{L^{2}} \begin{cases}
          R^{\frac 2q - \alpha } & \quad \text{ if } R<1,\\
          R^{\big ( 1- \frac 2q \big )\big (\frac 1p - \frac 12  \big) + \frac 1q  } & \quad \text{ if } R\ge1,
      \end{cases}
    \end{equation}
   where the constant $C_{\alpha,N}$ is given by \eqref{CN2}.
\end{lemma}
\begin{proof}
The proof is the same as the one of Lemma \ref{localizedlemmafree}: making use of estimates \eqref{est_norm_bes_nu} instead of \eqref{est_norm_bes} yields the result.
\end{proof}

As a consequence, condition \eqref{freerange} provides for $R<1$ the following additional limitation on the range of $q$:
    \begin{equation}\label{restrictionAB}
    \frac2q-\alpha >0 \Leftrightarrow q<\frac2\alpha.
    \end{equation}

    The rest of the proof works exactly as in the free case.
    \begin{remark}
    Condition \eqref{restrictionAB} provides an upper bound for the range of admissible Strichartz pairs in the Aharonov-Bohm case. Notice that if $\alpha\rightarrow 0$ (i.e. if the effect of the magnetic potential becomes negligible) then $q\rightarrow +\infty$ and we recover the full set of admissible pair. It can be proved that condition \eqref{restrictionAB} is indeed sharp (by retracing proof of Proposition 6.6 in \cite{cacdanzhayin1}), but we prefer not to insist on this here. Finally, notice that for $\alpha\in(1/2,1)$ the analogue of \eqref{restrictionAB} becomes $q<\frac2{1-\alpha}$, and this yields condition \eqref{q-alp} in Theorem \ref{ABtheorem}.
    \end{remark}

 \section{Dirac-Coulomb}

 We finally turn to the Dirac-Coulomb system. The punchline of the proof will be the same again, but now with some notable differences. The generalized eigenfunctions are indeed more complicated to estimate (the pointwise bounds will be derived in the appendix) and, more importantly, in order to provide dispersive estimates we shall need to localize on one side of the spectrum.

    \subsection{Preliminaries}\label{subsecDMprel}
The spectral theory for the Dirac-Coulomb operator has been widely investigated and it is now very well understood. The operator $\mathcal{D}_{DC}$ initially defined on $C^\infty_c(\R^3\backslash\{0\})$ is indeed essentially self-adjoint if and only if $|\nu|\leq \frac{\sqrt3}2$, and the domain of its self-adjoint extension is $H^1$; in the range $\frac{\sqrt3}2<|\nu|<1$ it admits a distinguished self-adjoint extension, that is characterised by the properties $D(\mathcal{D}_{DC})\subset D(|\mathcal{D}|^{1/2})$. In what follows we shall denote for brevity $\mathcal{D}_{DC}^{dist}=\mathcal{D}_{DC}$.
The spectrum of the operator has an essential part which is given by $\sigma_{ess}=(-\infty,-m]\cup[m,+\infty)$ plus a discrete part in the gap $(-m,m)$ consisting of eigenvalues given by Sommerfeld's celebrated fine-structure formula:
\begin{equation}\label{discspec}
E=E_n=\displaystyle{\rm sgn(\nu)} \frac{m}{\sqrt{1+\frac{\nu^2}{(k^2-\nu^2)+n_r^2}}},\qquad n_r=\begin{cases}
0,1,2,\dots {\rm if}\; k<0,\\
1,2,\dots {\rm if}\; k>0.
\end{cases}
\end{equation}
Notice thus that the discrete spectrum accumulates at $E=\pm m$ (depending on the sign of $\nu$) as $n_r\rightarrow+\infty $.
The literature on the topic is huge: we refer to \cite{nen}, the recent papers \cite{galmic}-\cite{estlewser2}-\cite{schisoltok}-\cite{dolestser} and references therein for details.

We now need to explicitly write the generalized eigenfunctions of the operator $\mathcal{D}_{DC}$.
To do so, resorting on the partial wave decomposition \eqref{spherdec} again (in this case the 3d version given by \eqref{decomp3}-\eqref{spherdec})
yields the following decomposition for the Dirac-Coulomb operator:
\begin{equation}\label{raddiropdc}
 \mathcal{D}_{DC}=
  \bigoplus_{k\in \mathbb{Z}^*}d_{k}^{DC}\otimes I_{4},
  \qquad
  d_k^{DC}=\left(\begin{array}{cc}m -\frac\nu r& -(\frac{d}{dr}+\frac{1}{r})+\frac kr \\(\frac{d}{dr}+\frac{}{r})+\frac kr & -m-\frac\nu r\end{array}\right), \qquad
  I_4:=
  \begin{pmatrix}
    I_2 & 0 \\
  0 & I_2
  \end{pmatrix}.
\end{equation}
Then, we are interested in solving the following equation
\begin{equation}\label{eq:genegDC}
d_{k}^{DC} \Psi_k(E, r)=d_{k}\left(\begin{matrix} \psi_k^1(E, r)\\ \psi_k^2(E, r)\end{matrix}\right)=E  \left(\begin{matrix} \psi_k^1(E, r)\\ \psi_k^2(E, r)\end{matrix}\right).
\end{equation}
The computations for the generalized eigenfunctions are quite involved but well understood, and one has the following solutions for $E>m$
(see e.g. \cite{lanlifrel} pag 112 \footnote{Notice that here we are interested in the repulsive case; we thus need to adapt the expressions obtained in \cite{lanlifrel} to the choice $\nu<0$.}): 
  \begin{equation}\label{psiDC}
    \Psi_{k}^{DC}(E, r)= \begin{pmatrix}  F_k^{DC}(E,r) \\
 G_k^{DC}(E,r)
\end{pmatrix}
  \end{equation}
  where
  \begin{equation}
\label{formula_F_k+}
F_k^{DC}(E;r) \coloneqq 2 N^+(E) e^{\frac{\pi \alpha_E}2} \frac{\lvert \Gamma (\gamma +1 +i\alpha_E) \rvert}{\Gamma (2\gamma +1)} \frac{(2{\bf p}r)^{\gamma- \frac 12}}{\sqrt r} Im \{ e^{ipr + i \xi(E)} {_1F_1} (\gamma -i \alpha_E, 2\gamma +1, 2i{\bf p}r) \}
\end{equation}
\begin{equation}
\label{formula_G_k+}
G_k^{DC}(E;r) \coloneqq 2 N^-(E)  e^{\frac{\pi \alpha_E}2} \frac{\lvert \Gamma (\gamma +1 +i\alpha_E) \rvert}{\Gamma (2\gamma +1)} \frac{(2{\bf p}r)^{\gamma- \frac 12}}{\sqrt r} Re \{ e^{ipr + i \xi(E)} {_1F_1} (\gamma -i \alpha_E, 2\gamma +1, 2i{\bf p}r) \}
\end{equation}
with
\[
\alpha_E \coloneqq \frac{\nu E}{\bf p}, \quad \gamma=\sqrt{k^2-\nu^2},\quad e^{2i\xi(E)}=\frac{k-\frac{i\alpha_E}E m}{\gamma-i\alpha_E}=\frac{k-\frac{i\nu m}{\bf p}}{\gamma-i\alpha_E}m,\quad {\bf p}=\sqrt{E^2-m^2}
\]
and $N^\pm(E)$ given by \eqref{Enorm}.

We are again in position to define the relativistic Hankel transform for the Dirac-Coulomb model, that is the following
\begin{definition}
Let $f(r)=\begin{pmatrix} f_1(r) \\
f_2(r)
\end{pmatrix}\in L^2(\R^3,\C^4)$. For $k\in\Z^*$ we define the following sequence of operators
\begin{equation}\label{DChank}
\mathcal P_k^{DC,+} f(E) \coloneqq \int_0^{+\infty} \psi^{DC}_k(E,r)^T f(r)r^2 dr
\end{equation}
where $\psi^{DC}_k$ is given by \eqref{psiDC}. 
\end{definition}

The following Proposition is proved in \cite{git} using the Krein method of guiding functionals (see {\it e.g. \cite{AG63}} for an introduction to this method).

\begin{proposition}\label{dchankprop}\cite{git}
Let $f\in Dom(d_k^{DC})$. The operator 
$$\mathcal P_k^{DC,+} \colon [L^2((0,+\infty), r^{2} dr)]^2 \rightarrow L^2 ((m,+\infty), E dE)$$ defined by \eqref{DChank} satisfies for any $k\in\Z$ the following properties:
\begin{enumerate}
\item $\mathcal P_k^{DC,+}$ is invertible in the sense that the inverse operator $\mathcal (P_k^{DC,+})^{-1} \colon L^2((m,+\infty), E dE) \rightarrow [L^2((0,+\infty), r^2 dr )]^2$ defined by
\begin{eqnarray}\label{dcinv}
\mathcal (P_k^{DC,+})^{-1} \phi (r') 
&=&\int_{(m,+\infty)} \begin{pmatrix}
F_k^{DC}(E,r) \phi(E)&\\
-i G_k^{DC}(E,r)\phi(E)
\end{pmatrix} 
EdE
\end{eqnarray}
 where $F_k^{DC}$ and $G_k^{DC}$ are given by \eqref{formula_F_k+}-\eqref{formula_G_k+} is such that $$(P_k^{DC,+})^{-1}(P_k^{DC,+})\Pi_+f(r)=\Pi_+f(r)$$ for any $f\in Dom(d_k^{DC})$.
\item $\mathcal{P}^{DC,+}_k$ is an isometry.
\item $(\mathcal P_k^{DC,+} d_k^{DC} f) (E)= E \mathcal P_k^{DC,+}f(E)$.
\end{enumerate}
\end{proposition}

\begin{remark}
Notice that, compared with the analogous results for the free and the AB cases, {\it i.e.} Propositions \ref{freehankprop}-\ref{ABankprop}, the result above is slightly less ambitious, as it does not provide a complete ``reconstruction formula'' for the Dirac-Coulomb operator. Such a formula is provided in \cite{git}, but we only consider the case of functions which have energies supported in the upper-half of the spectrum, as this is what we shall ultimately need.
\end{remark}

\subsection{Generalized Strichartz estimates (proof of Theorem \ref{DCtheorem})}
We consider the Cauchy problem
 \begin{equation}\label{eqDC}
        \begin{cases}
        i \partial_t u + \mathcal{D}_{DC} u = 0,\qquad  u \colon \R_t \times \R^3_x \rightarrow \C^4,\\
        u(0,x) = u_{0}.
        \end{cases}
    \end{equation}
    with $\mathcal{D}_{DC}$ defined by \eqref{DCoperator} with $\nu<0$, and where we assume that $u_0$ is in the range of $\Pi_+:=\chi_{[0,+\infty)}(\mathcal{D}_{DC})$.

    Relying on decomposition \eqref{spherdec}, on the operator $\mathcal{P}_k^{DC}$ as defined in \eqref{DChank} and on Proposition \ref{dchankprop}, we can represent the solution as 
    \begin{equation*}
e^{it|\mathcal{D}_{DC}|}u_0=\sum_{k\in\Z^*}u_k(t,r)\cdot \Xi_k(\theta)=
\sum_{k\in\Z^*}e^{it|d_k^{DC}|}u_{0,k}\cdot \Xi_k(\theta)=\sum_{k\in\Z^*}(\mathcal{P}_k^{DC,+})^{-1}[e^{itE} (\mathcal{P}_k^{DC,+} u_{0,k})(E)](r) \cdot \Xi_k(\theta)
 \end{equation*}
 where $d_k^{DC}$ is given by \eqref{raddiropdc} and $|d_k^{DC}|=d_k^{DC}\Pi_+$, and the initial condition is $u_0(r)=\begin{pmatrix}
       u_0^1(r)\\
       u_0^2(r)
    \end{pmatrix}\in L^2((r^2dr)^2)$. 
This allows for an explicit representation of the components of the vector $ u_k (t,r) =\begin{pmatrix}
       u^1_k(t,r)\\
       u^2_k(t,r)
    \end{pmatrix}$
   in terms of the ones of the generalized eigenfunctions $\psi_k^{DC}$, analogous to \eqref{freeu1}-\eqref{freeu2}. 
   
   Then, the proof follows the very same lines as the previous cases; without any modification we indeed get to the equivalent of estimate \eqref{estII} that here now is:

\begin{equation}\label{checkpointdc}
       \lVert I^+_{k,DC} \rVert_{L^p_t L^q_{r dr}}^2  \le 
       \sum_{R\in 2^\Z}  \bigg ( \sum_{N \in 2^\Z}   {N^{\frac52-\frac3q}}  \Big \lVert   \int_0 ^\infty H_k^{DC}(N,y,\rho)  e^{it \sqrt{(Ny)^2 +m^2}} f(\sqrt{(Ny)^2 +m^2}) \phi ( y) y dy \Big \rVert_{L^p_t L^q_{\rho^2d\rho}[NR,2NR]} \bigg ) ^2
       \end{equation}
 where $f \in L^2 ((m,+\infty), EdE)$, $E= \sqrt{{\bf p}^2 +m^2}$ and $H_k^{DC}(N,y,\rho)$ is either $F_k^{DC}$ or $G_k^{DC}$ \eqref{formula_F_k+}-\eqref{formula_G_k+} rescaled as follows:
 \begin{equation}\label{F_krescaled}
F_k^{DC}(N, y, \rho)= N^+(E)e^{\frac{\pi \alpha_{Ny}}2} \frac{\lvert \Gamma (\gamma +1 +i\alpha_{Ny}) \rvert}{\Gamma (2\gamma +1)} \frac{(2 y \rho)^{\gamma-1/2}}{\sqrt \rho} Im\{e^{ipr + i  \xi(N)} {_1F_1} (\gamma -i \alpha_{Ny}, 2\gamma +1, 2iy\rho) \},
\end{equation}
 \begin{equation}\label{G_krescaled}
G_k^{DC}(N, y, \rho)= N^-(E)e^{\frac{\pi \alpha_{Ny}}2} \frac{\lvert \Gamma (\gamma +1 +i\alpha_{Ny}) \rvert}{\Gamma (2\gamma +1)} \frac{(2 y \rho)^{\gamma-1/2}}{\sqrt \rho} Re\{e^{ipr + i  \xi(N)} {_1F_1} (\gamma -i \alpha_{Ny}, 2\gamma +1, 2iy\rho) \}
\end{equation}
    where 
\[
\alpha_{Ny} \coloneqq \frac{\nu \sqrt{(Ny)^2 + m^2}}{Ny}, \quad e^{2i  \xi(N)}=\frac{- k+\frac{i\nu m}N}{\gamma-i\alpha_{Ny}}
\]
and with the normalization terms $N^\pm$ given by \eqref{Enorm} (which, as in the previous sections, we shall neglect in our forthcoming estimates).

Then, the key step is represented   by the following result:

\begin{lemma}\label{localizedlemmaDC}
Let $h \in C_c^\infty$ with $\supp h \subset [\frac12, 1] = \colon I$, and let $k\in\Z^*$. Then, for any $(p,q)\in(2,+\infty)\times(2+\infty)$ and $N\in 2^{\Z}$ the following estimate holds 
    \begin{equation}\label{locestdc}
      \Big \lVert   \int_0 ^\infty H_k^{DC}(N,y,\rho) e^{it \sqrt{(Ny)^2 +m^2}} h(y) dy \Big \rVert_{L^p_t L^q_{\rho^{2} d\rho}[R,2R]}\leq C_N^{\frac 2p}  \|h\|_{L^{2}} \begin{cases}
          {R^{\gamma-1+\frac 3q}} & \quad \text{ if } R<1,\\
          R^{\frac 1q + \big (1- \frac 2q \big ) \beta(p)} & \quad \text{ if } R\ge1,
      \end{cases}
    \end{equation}
    where $H_k^{DC}$ is either $F_k^{DC}$ or $G_k^{DC}$ as given by \eqref{F_krescaled}-\eqref{G_krescaled}, with
    \begin{equation}\label{beta}
    \beta(p)=\begin{cases}
\frac1p-1\qquad {\rm if}\: p\in[2,4),\\
\frac1p-\frac56\qquad {\rm if}\: p\in[4,+\infty)
\end{cases}
\end{equation}
   and where the constant
    \begin{equation}\label{CN2bis}
    C_N\sim\begin{cases}
    N^{-\frac 12} \quad \text{if }\: N=2^n,\\
    N^{-1}  \quad \text{if} \: N=2^{-n}
    \end{cases}
    \end{equation}
    is independent on $\nu$.

\end{lemma}

\begin{proof}
The proof follows the very same lines of the one of Lemma \eqref{localizedlemmafree}: we simply need to replace the ``standard'' estimates on Bessel functions given by Lemma \ref{est_norm_bes_nu} with the corresponding ones on the functions $F_k$, $G_k$, which are given in Corollary \ref{finalcorDC} and are consequences of Theorem \ref{thm_formuladc}. The remaining details are straightforward.
\end{proof}

Here again, we deal separately with the cases $k=\pm1$ (the ``radial part'') and  $k\neq\pm1$ by making use of the projectors $P_{rad}$ and $P_\bot$ as defined in \eqref{def-pro3}.
Recalling \eqref{spherdec}, for any initial datum $u_0=\sum_{k\in\Z^*} u_{0,k}(r)\cdot \Xi_k(\theta)\in D(\mathcal{D}_{DC})$ we can thus decompose the flow as follows:
 \begin{equation*}
e^{it\mathcal{D}_{DC}} u_0=e^{it\mathcal{D}_{DC}}P_{rad} u_0+e^{it\mathcal{D}_{DC}}P_{\bot} u_0
  \end{equation*}
 where now
 \begin{equation*}
P_{rad} u_0= u_{0,1}(r)\cdot \Xi_1(\theta)+u_{0,-1} (r)\cdot \Xi_{-1}(\theta)
 \end{equation*}
and
\begin{equation*}
P_\bot u_0=\sum_{k\in\Z^*\backslash\{\pm1\}} u_k(r)\cdot \Xi_k(\theta).
\end{equation*}
\\\\

$\bullet$ {\bf Strichartz estimates for $e^{it\mathcal{D}_{DC}}P_{\bot} u_0$.} Picking up from estimate \eqref{checkpointdc}, we retrace the proof developed in previous subsections: we only need to use Lemma \eqref{localizedlemmaDC} instead of \eqref{localizedlemmafree}. Notice that if $k\neq \pm1$ then $\gamma-1>0$, and the generalized eigenfunctions are regular in the origin. Therefore, the full set of Strichartz estimates are retrieved.
\\\\
$\bullet$ {\bf Strichartz estimates for $e^{it\mathcal{D}_{DC}}P_{rad} u_0$.} Let us consider the cases $k=1$ and $k=-1$ (individually). In this sector the singularity of the generalized eigenfunctions come into play: following the proof of the free case again, in the application of Lemma \eqref{locestdc} we need to impose the condition
$$
\gamma_1-1+\frac3q>0 \Rightarrow q <q(\nu)=\frac3{1-\sqrt{1-\nu^2}}
$$
where we set $\gamma_1 \coloneqq \sqrt{1 - \nu^2}$.
The rest of the proof follows.


\appendix
\section{Auxiliary estimates}
In this appendix we collect some auxiliary estimates on special functions that we have used through the paper.

\subsection{Estimates on Bessel functions}

The first result we provide is some pointwise estimates on Bessel functions which are uniform with respect to the order; these are needed in order to prove generalized Strichartz estimates for the free case and the Aharonov-Bohm magnetic potential.
The result is the following:

\begin{lemma}\label{lembessel}
    Let  $R \in 2^\Z$. Then, for all $q \in [2,+\infty) $ and $\nu \ge 2$
    \begin{equation}
    \label{est_norm_bes}
    \lVert \rho^{-\frac {n-2}2} J_{\nu}(\rho) \rVert_{L^q[R,2R]} \le c \begin{cases}
         R^{\frac 1q}, &\quad \text{ if } R <1   \\
         R^{\beta(q)}, &\quad \text{ if } R \ge 1 
    \end{cases}
    \end{equation}
    \[
     \lVert (\rho^{-\frac {n-2}2} J_{\nu}(\rho))' \rVert_{L^q[R,2R]} \le \tilde c \begin{cases}
         R^{\frac 1q -1 }, &\quad \text{ if } R <1  \\
         R^{\beta(q)}, &\quad \text{ if } R \ge 1 
    \end{cases}
    \]
    where
    $$
\beta(q)=
\begin{cases}
\frac1q-\frac{n-1}2\qquad {\rm if}\: q\in[2,4),\\
\frac1q-\frac{3n-4}6\qquad {\rm if}\: q\in[4,+\infty)
\end{cases}
    $$
  and  where the constants $c, \tilde c$ depend on $q$ but are independent on $\nu$. 
  
  Moreover, for all $q \in [2, + \infty)$ and for any fixed $\nu \in \R$ we have 
  \begin{equation}
    \label{est_norm_bes_nu}
    \lVert \rho^{-\frac {n-2}2} J_{\nu}(\rho) \rVert_{L^q[R,2R]} \le c_\nu \begin{cases}
         R^{\frac 1q + \nu -\frac {n-2}2}, &\quad \text{ if } R <1   \\
         R^{\frac 1q  - \frac{n-1}2}, &\quad \text{ if } R \ge 1 
    \end{cases}
    \end{equation}
    \[
     \lVert (\rho^{-\frac {n-2}2} J_{\nu}(\rho))' \rVert_{L^q[R,2R]} \le \tilde c_\nu \begin{cases}
         R^{\frac 1q + \nu -\frac n2 }, &\quad \text{ if } R <1  \\
         R^{\frac 1q  - \frac{n-1}2}, &\quad \text{ if } R \ge 1. 
    \end{cases}
    \]
  
\end{lemma}

\begin{proof}
For the sake of brevity we will write down only the proof of the estimates \eqref{est_norm_bes}, \eqref{est_norm_bes_nu}. The ones for the first derivative follows from the relation 
\[
J_\nu'(\rho) = J_{\nu -1}(\rho) - \nu \frac{J_\nu(\rho)}\rho
\]
which holds for every $\nu \in \R$ and $\rho \in (0,+\infty)$. \\
In order to prove this result, we rely on the following well established estimate for Bessel functions, that holds for any $\nu\geq2$:
\begin{eqnarray}\label{estbes}
|J_\nu(r)|\leq C\times
\begin{cases}
e^{-D\nu},\qquad\qquad\qquad\qquad\qquad\; 0<r\leq \nu/2,\\
\nu^{-1/4}(|r-\nu|+\nu^{1/3})^{-1/4},\qquad \nu/2<r\leq 2\nu,\\
r^{-1/2},\qquad\qquad\qquad\qquad\qquad\; 2\nu<r
\end{cases}
\end{eqnarray}
 for some positive constants $C$ and $D$ independent of $r$ and $\nu$ (for this estimate see e.g. \cite{cor}-\cite{barcor}).  We split the proof in two cases: $ R \ge1$ and $ R < 1$. \\
$i) \, R \ge 1$: We write the integral of integration as $[R,2R] = I_1 \cup I_2 \cup I_3$, where
\[
I_1 = [R,2R] \cap [0, \tfrac{\nu}2], \quad  [R,2R] \cap [ \tfrac{\nu}2, 2\nu], \quad I_3 =  [R,2R] \cap [2\nu,+ \infty)
\]
and we estimate each interval separately. \\
For $I_1$ we can assume $2R \le \nu$ otherwise $I_1 = \emptyset$. Let $q \in [2,+\infty)$, then
\[
\int_{I_1} \lvert \rho^{- \frac{n-2}2} J_\nu (\rho) \rvert ^q d\rho \lesssim \int_{I_1} \rho^{- \frac{(n-2)q}2} e^{-D\nu q} d\rho \le C_\alpha R^{-\alpha} \quad \forall \alpha >0.
\]
For $I_2$ we can assume $\tfrac{\nu}4 \le R \le 2 \nu$, otherwise $I_2 = \emptyset$. Let $q \in [2,4)$, we compute the integral and obtain
\begin{equation}
\label{est_I_2_improved}
\begin{split}
\int_{I_2} \lvert \rho^{- \frac{n-2}2} J_\nu (\rho) \rvert^q dr&  \lesssim  \int_{\frac{\nu}2}^{2\nu}  \rho^{- \frac{(n-2)q}2} \nu^{-\frac{q}4}\big ( \big \lvert  \nu - \rho \big \rvert + \nu ^{\frac13} \big ) ^{- \frac q4} d\rho \\
& = \nu^{- \frac{(2n-3)q}4 +1 } \int_{\frac 12}^2 \big (\nu^\frac13 (\nu^\frac 23 \lvert 1-\tau \rvert +1) \big)^{- \frac q4} d\tau  \\
& = \nu^{- \frac{(2n-3)q}4 +1 - \frac q{12}} \int_{-\frac12}^1 \big ( \nu^{\frac23} \lvert y \rvert +1 \big )^{-\frac q4} dy \\
& = \nu^{- \frac{(2n-3)q}4 +1- \frac q{12} - \frac23} \int_{- \frac{\nu^{\frac23}}2}^{\nu^{\frac23}} \big ( 1 + \lvert x \rvert )^{-\frac q4} dx \\ 
& = \nu^{- \frac{6n -8}{12} q + \frac 13} \frac 4{4-q} \bigg [ \big ( 1 + \tfrac{\nu^{\frac23}}2 \big )^{1 - \frac q4} + \big ( 1 + \nu^{\frac23} \big )^{1 - \frac q4} -2 \bigg ] \\
& \lesssim \nu^{- \frac{6n -8}{12} q + \frac 13 + \frac 23 - \frac q6} = \nu^{1 - \frac {n-1}2q} \simeq R^{1 - \frac {n-1}2q} .
\end{split}
\end{equation}
For $q \in [4, + \infty)$ we estimate the norm as
 \begin{equation}
 \label{est_I2}
 \int_{I_2} \rvert\rho^{- \frac{n-2}2} J_\nu (\rho) \rvert^q d\rho  \lesssim \int_{I_2} \nu^{- \frac{2n-1}2 q} \big ( \big \lvert  \nu - r \big \rvert + \nu ^{\frac13} \big ) ^{- \frac q4} \rho^{\frac q2} d\rho \lesssim \nu^{- \frac{2n-1}2q - \frac 1{12}q} \int_R^{2R} r^{\frac{q}2} \lesssim R^{1 - \frac{3n -4}6 q}.
 \end{equation}
 Lastly, for $I_3$ we get 
 \[
 \int_{I_3} \lvert \rho^{- \frac{n-2}2} J_\nu (\rho) \rvert^q d\rho \lesssim \int_R^{2R} \rho^{- \frac {n-1}2q} d\rho \lesssim R^{ 1 - \frac {n-1}2q }.
 \]
 $ii)\,  R < 1$:
In order to prove the estimate for $R <1$, let us state a refined version \eqref{estbes} that holds for any $\nu \ge 2$ and $\rho \le1$. Let us recall the Taylor expansion of the Bessel functions
\[
J_\nu(\rho) = \Big ( \frac \rho2 \Big) ^\nu \sum_{k=0}^\infty \frac{(-1)^k}{k! \Gamma (\nu + k +1)} \Big ( \frac \rho2 \Big)^{2k}. 
\]
Therefore, if $\rho \le 1$ 
\[
\lvert J_\nu (\rho)\rvert \le \Big ( \frac \rho2 \Big) ^\nu \sum_{k=0}^\infty \frac 1{k! 2^{2k}} \frac 1{|\Gamma(\nu +k +1)|}.
\]
 By Stirling's approximation, if $k \gg 1$
 \[
 \frac 1{|\Gamma(\nu +k +1)|} \sim (\nu +k +1) \Big ( \frac e{\nu +k +1} \Big )^{\nu +k +1}.
 \]
 Therefore, there exists a constant independent of $\nu $ such that 
 \[
 \lvert J_\nu(\rho) \rvert \le c e^{-\nu}\rho^\nu 
 \]
 for any $\rho \le 1$.  Let $q \in [2, +\infty)$, we observe that $r \in [R, 2R]$, $R=2^{-n}, n >0$ implies $r \le 1$. Therefore, we have
 \[
 \bigg (\int_R^{2R} \lvert \rho^{- \frac{n-2}2} J_\nu (\rho) \rvert ^q d\rho \bigg)^{\frac 1q} \lesssim R^{\frac 1q}.
 \]
 
 
 The estimates we proved hold uniformly in  $\nu$, for $\nu \ge 2$. However, both in the free and Aharonov-Bohm cases we will need similar estimates for a finite number of Bessel functions with indices $\nu \in [-\tfrac 12, 2)$. We recall that for any fixed $\nu \in \R$ the following holds
 \[
 \begin{split}
 \lvert J_\nu (\rho) \rvert \le c \rho^\nu \quad \text{if } \rho \le 1 \quad \text{and} \quad \lvert J_\nu(\rho) \rvert \le c\rho^{-\frac12} \quad \text{if } \rho \ge 1.
 \end{split}
 \]
Therefore, for any finite $q \in [2, +\infty)$, we get that if $R \ge 1$
 \[
 \int_R^{2R} |\rho^{- \frac{n-2}2 } J_\nu (\rho)|^q d\rho \le c_1 R^{1 - \frac{n-1}2 q} 
 \]
 and the same holds for the first derivative. Instead, if $R <1$ we have
 \[
 \int_R^{2R} |\rho^{- \frac{n-2}2 } J_\nu (\rho)|^q d\rho \le c_2 R^{1 + \nu q - \frac{n-2}2q}
 \]
 and
 \[
 \int_R^{2R} |(\rho^{- \frac{n-2}2 } J_\nu (\rho))'|^q d\rho \le c_3 R^{1 + \nu q - \frac n2q}.
 \]
 Let us underline that here the constants $c_i$, $i=1,2,3$ depend on $\nu$. 
\end{proof}

\subsection{Estimates on Confluent hypergeometric functions.}\label{confappendix} The pointwise estimates for the generalized eigenfunctions of the Dirac-Coulomb operator are significantly more delicate. We prove the following

\begin{theorem}
\label{thm_formuladc}
Let $\nu \in [-1, 0)$. Consider the function 
$$\Phi_k^{DC}(N,\rho) = (iF_k^{DC} + G_k^{DC})(N,\rho)$$ with $\rho \in (0,+\infty)$ and $N \in \{2^n \}_{n \in \N} \cup \{2^{-n}\}_{n \in \N*}$ as defined in \eqref{completeeigendc} (we are taking $y=1$). 
Then, for all $\lvert k \rvert \ge 2$, there exist positive constants $C\,,\,D$ independent of $k\,,\,\nu$ such that the following pointwise estimates hold for all $\rho\in (0,+\infty)$: 
\begin{equation}\label{esteigformuladc1}
|\Phi_k^{DC}(N,\rho)| \leq 
\begin{cases}
(\min\{\rho/2\,,\,1\})^{\gamma-1}e^{-D\vert k \vert}\,,\qquad\quad\, 0<\rho \leq \max\{\vert k \vert/2\,,\,2\},\\
\vert k \vert^{-\frac34}\big(|\,\vert k \vert-\rho \,|+\vert k \vert^{\frac13}\big)^{-\frac14},\quad\quad \frac{\vert k \vert}2\leq \rho\leq 2\vert k \vert,
\\
\rho^{-1}
,\qquad\qquad\qquad\qquad\qquad\quad\quad\quad\;\,\, \rho \geq 2\vert k \vert,
\end{cases}
\end{equation}
and
\begin{equation}\label{esteigformuladc2}
|(\Phi_k^{DC})'(N,\rho)| \leq 
\begin{cases}
(\min\{\rho/2\,,\,1\})^{\gamma-2}e^{-D\vert k \vert}\,,\qquad\quad\, 0<\rho\leq \max\{\vert k \vert/2\,,\,2\},\\
\vert k \vert^{-\frac34}\big(|\,\vert k \vert-\rho\,|+\vert k \vert^{\frac13}\big)^{-\frac14},\quad\quad \frac{\vert k \vert}2\leq \rho\leq 2\vert k \vert,
\\
\rho^{-1}
,\qquad\qquad\qquad\qquad\qquad\quad\quad\quad\;\,\, \rho \geq 2\vert k \vert.
\end{cases}
\end{equation}
Moreover, for $k=\pm1$ the following estimates hold.
     \begin{equation}
     \label{esteigformuladcrad1}
     \lvert \Phi_{\pm1}^{DC} (N,\rho) \rvert \le c_{\gamma_1} 
     \begin{cases}
         \rho^{\gamma_1 - 1}, \quad &\rho \le 1,\\
         \rho^{-1}, \quad   &\rho > 1
     \end{cases}
      \end{equation}
     and
     \begin{equation}
     \label{esteigformuladcrad2}
     \lvert (\Phi_{\pm1}^{DC})' (N,\rho) \rvert \le c_{\gamma_1} 
     \begin{cases}
         \rho^{\gamma_1 - 2}, \quad & \rho \le 1,\\
         \rho^{-1}, \quad   &\rho > 1
     \end{cases}
     \end{equation}
   where $\gamma_1 \coloneqq \sqrt{1 - \nu^2}$.
\end{theorem}

\begin{proof}
The estimate in the massless case is proved in \cite{cacserzha}, Theorem 1.1. The presence of a mass term $m>0$ makes the analysis more sophisticated, as indeed the dependence of the function from the variables $E$ and $\rho$ is not the same here; we shall thus briefly retrace the proof developed in \cite{cacserzha}, and highlight where the main differences appear. 

We shall focus only on the proof of \eqref{esteigformuladc1}, as the one of \eqref{esteigformuladc2} can be obtained with minor modifications.

As a starting point, let us recall the following expression for $\Phi_k^{DC}(N,\rho)$:
\begin{equation}\label{completeeigendc}
\Phi_k^{DC}(N,\rho)=( G_k^{DC}+ iF_k^{DC}) (N, \rho)= e^{\frac{\pi \alpha_{N}}2} \frac{\lvert \Gamma (\gamma +1 +i\alpha_{N}) \rvert}{\Gamma (2\gamma +1)} \frac{(2  \rho)^\gamma}\rho e^{i\rho + i \xi(N)} {_1F_1} (\gamma -i \alpha_{N}, 2\gamma +1, 2i\rho) 
\end{equation}
where 
\[
\alpha_N \coloneqq \frac{\nu \sqrt{N^2 + m^2}}{N}, \quad \gamma=\sqrt{k^2-\nu^2},\quad e^{2i  \xi(N)}=\frac{- k+\frac{i\nu m}N}{\gamma-i\alpha_N}.
\]
Notice that if $m=0$ then $\alpha_N=\nu$, which allows to significantly simplify the expression above; as a matter of fact, the term $\alpha_N$ will indeed encode all the differences between the massless and massive cases. It is important to stress that, in our repulsive setting $\nu<0$, the term $\alpha_N<0$; this will be a key fact in our estimates.

We recall the following integral representation (see \cite{abram} pag.505):
\begin{equation}\label{eigenintraprep}
\Phi_k^{DC}(N,\rho) = \frac{|\Gamma(\gamma+1+i\alpha_N)|}{\Gamma(\gamma+1+i\alpha_N)}\frac{e^{\frac{\pi\alpha_N}2}e^{i  \xi}{\rho}^{\gamma-1}
}{2^{\gamma+1/2}\Gamma(\gamma-i\alpha_N)}
\int_{-1}^1e^{i\rho t}(1+t)^{\gamma-1-i\alpha_N}(1-t)^{\gamma+i\alpha_N} dt
\,.
\end{equation}
We can write
\begin{equation}\label{moduleigen}
|\Phi_k^{DC}(N,\rho)|= \frac{e^{\frac{\pi\alpha_N}2}\rho^{\gamma-1}
}{2^{\gamma+1/2}|\Gamma(\gamma-i\alpha_N)|}
|I_{\gamma,\rho,N}|
\end{equation}
with
\begin{equation}\label{integral}
I_{\gamma,\rho,N}=\int_{-1}^1 e^{i\rho t}(1+t)^{\gamma-1-i\alpha_N}(1-t)^{\gamma+i\alpha_N} dt =\int_{-1}^1 e^{-i\rho t}(1-t)^{\gamma-1-i\alpha_N}(1+t)^{\gamma+i\alpha_N} dt \,.
\end{equation}
Now, it is convenient to analyze separately the cases of ``large'' and ``small'' energies, that is $N=2^n$ and $N=2^{-n}$, $n \in \N$. \\\\
$\bullet$ {\bf Case $N=2^n$.} Notice that in this case $\langle m \rangle \nu \le \alpha_N \le \nu$. Therefore the proof of \eqref{esteigformuladc1} strictly follows the one of Theorem 1 in \cite{cacserzha}.
\\\\
$\bullet$ {\bf Case $N=2^{-n}$.} This situation is more subtle, as indeed getting closer to the ``threshold energy'' $E=m$ provides additional difficulties. In this case, the term $\alpha_N$ is still bounded from above but it is not bounded from below any more. In particular, we have that $\alpha_N \sim_{n \to + \infty}- N^{-1}$. Therefore, in order to prove our desired estimate, we will need to adapt the strategy developed for the massless case, providing new bounds on the terms which depend on $N$. Let us give more details.\\

We start by estimating the prefactor $\frac{e^{\frac{\pi\alpha_N}2}\rho^{\gamma-1}
}{2^{\gamma+1/2}|\Gamma(\gamma-i\alpha_N)|}$ which appears on the RHS of formula \eqref{moduleigen}. Stirling's formula yields
\[
\frac 1{\Gamma(\gamma -i \alpha_N)} = \frac{ \sqrt{\gamma -i \alpha_N}}{\sqrt{2\pi}} \frac{e^{\gamma -i \alpha_N}}{(\gamma - i\alpha_N)^{\gamma -i\alpha_N}}\quad \text{ for } \gamma,n \gg 1.
\]
Therefore, if $\gamma \ge \frac 12$, we have 
\[
\begin{split}
\frac 1{|\Gamma(\gamma -i \alpha_N)|} \le \frac 1{\sqrt{2\pi}} \frac{e^\gamma}{| \gamma - i \alpha_N|^{\gamma - \frac 12}} e^{- \alpha_N Arg (\gamma -i\alpha_N)} \le C \frac{e^\gamma}{|k|^{\gamma - \frac 12}}e^{\frac \pi 2 |\alpha_N|}.
\end{split}
\]
We thus obtain the following estimate
\begin{equation}
    \label{est_prefactor}
    \frac{e^{\frac{\pi\alpha_N}2}\rho^{\gamma-1}
}{2^{\gamma+1/2}|\Gamma(\gamma-i\alpha_N)|} \le C \bigg ( \frac{e}{2|k|} \bigg)^{\gamma- \frac 12} \rho^{\gamma -1}
\end{equation}
where the constant $C$ does not depend on $k,N$. We observe, however, that if $|k|=1 $ and $\nu < - \frac{\sqrt 3}2$ then $\gamma < \frac 12$. In this case, $
\lvert \gamma-i \alpha_N \rvert = \sqrt{\alpha_N^2 + \gamma_1^2} = N^{-1} \sqrt{ \nu^2 m^2 + N^2}$, and thus
\begin{equation}
\label{est_gamma_alpha_n}
c_{\nu,m} N^{-1} \le \lvert \gamma_1-i \alpha_N \rvert  \le C_{\nu,m} N^{-1}.
\end{equation}
Therefore, the prefactor can be estimated as
\begin{equation}
    \label{est_prefactor_2}
 \frac{e^{\frac{\pi\alpha_N}2}\rho^{\gamma_1-1}
}{2^{\gamma_1+1/2}|\Gamma(\gamma_1-i\alpha_N)|} \le C_{\gamma_1} N^{ \gamma_1-\frac 12 } \rho^{\gamma_1 - 1}.   
\end{equation}
We now need to bound the modulus of the integral $I_{\gamma,\rho,N}$ given by \eqref{integral} that we rewrite as
\begin{equation}
\label{integral_phase_ampl}
I_{\gamma, \rho,N} = \int_{-1}^1 g(t) e^{\rho h_q(t) dt}
\end{equation}
with 
\[
q = \frac{\gamma-1}\rho,\quad g(t)= (1+t)^{1+i\alpha_N}(1-t)^{-i\alpha_N} \quad \text{ and } \quad h_q(t)=-it + q \ln (1-t^2).\\
\]
It is important to stress the fact that the function $h_q(t)$ does not depend explicitly on $E$ (that is, on $N$): this allows to retrieve most of the computations developed in the massless case in \cite{cacserzha}. On the other hand, the function $g(t)$ does depend on $E$: as a consequence, in the application of the Laplace's Method, we will need some additional estimates uniform with respect to $N$ on the function $g$.

We deal with three different sectors separately.

\textbf{Sector $0< \rho \le \max (\frac {|k|}2,2)$.} 
Let $\gamma \ge \frac 12$. We observe that $I_{\gamma,\rho,N}$ is uniformly bounded. Indeed,
\[
\lvert I_{\gamma,\rho,N} \rvert \le \int_{-1}^1 (1-t^2)^{\gamma -1}(1-t)^1 dt \le 2.
\]
Using \eqref{est_prefactor}, if $2 \le \rho \le \frac{|k|}2$, we get that $\lvert \Phi_k \rvert \le C \big ( \frac e4 \big)^\gamma \le C e^{-D|k|}$. This covers the case of large $k$. 
Then, if $0 < \rho < 2$ and $\lvert k \rvert\le 4$, we can write an estimate in the form $\lvert \Phi_k \rvert \le C e^{-D|k|} \big ( \frac \rho2 \big)^{\gamma -1}$. The combination of these gives the first estimate in \eqref{esteigformuladc1} for any $k$.\\
Let now $\gamma < \frac 12$, then we need a finer estimate of $I_{\gamma, \rho,N}$ in order to compensate the factor $N^{\gamma_1 - \frac 12}$ in \eqref{est_prefactor_2}. To do so, we split the integral 
\[
I_{\gamma_1, \rho,N} = \int_{-1}^1 e^{-i\rho t} (1-t)^{\gamma_1-1 - i \alpha_N} (1+t)^{\gamma_1+ i \alpha_N }dt = \int_{-1}^0 \dots dt + \int_0^1 \dots dt = I_- + I_+.
\]
We estimate each term separately; we observe that 
\[
I_- = \int_{-1}^0  \Big ( \frac{1}{i \alpha_N + \gamma_1 + 1} \Big ( \frac {1+t}{1-t} \Big )^{\gamma_1 +1 + i \alpha_N}\Big )' \frac{e^{-i\rho t}}2 (1-t)^{2\gamma_1 + 1} dt.
\]
By integration by parts, we get
\[
\lvert I_- \rvert \le \frac 1{2 \lvert i \alpha_N + \gamma_1 + 1\rvert } \Big [ 1 + \int_{-1}^0 \Big ( \frac{1+t}{1-t}\Big )^{\gamma_1+1} (1-t)^{2\gamma_1 } \lvert i \rho(1-t) + 2\gamma_1 + 1\Big  ] \le C N
\]
where for the last inequality we recall the estimate \eqref{est_gamma_alpha_n} and $\rho < 2$. We can estimate $\bar I_+$ similarly, writing $ \frac{-2}{(1+t)^2} \big ( \frac{1-t}{1+t} \big )^{i \alpha_N + \gamma -1} = \Big ( \frac 1{i \alpha_N + \gamma}\big ( \frac{1-t}{1+t} \Big )^{i \alpha_N + \gamma} \Big )' $. Summing up, we get
\[
\lvert \Phi_k \rvert \le c_{\gamma_1} N^{\gamma_1 + \frac 12} \rho^{\gamma_1 -1} \le  c_{\gamma_1} \rho^{\gamma_1 -1}.
\]

\textbf{Sector $ \rho \ge \max (2,\frac {(\gamma +1)^2}2)$.} 
We start by deforming the interval of integration $(-1,1)$ into the bounded contour $(-1, -1 -iA] \cup [-1 -iA, 1-iA] \cup [1-iA,1)$ for $A>0$. We then pass to the limit for $A \to + \infty$ and we obtain
\[
\int_{-1}^1 g(t) e^{\rho h_q(t)} dt = \int_{-1}^{-1 -i\infty} g(z) e^{\rho h_q(z)} dz + \int_{1-i\infty}^1 g(z) e^{\rho h_q(z)} dz =\colon I^{(-1)} + I^{(1)}.
\]
Indeed, recalling that $\alpha_N <0$ and $Arg \big ( \frac{1-z}{1+z} \big ) > 0$ the function $g(z)e^{\rho h_q(z)}$ decays exponentially when $\Im(z) \to - \infty$, therefore
\[
\lim_{A \to + \infty} \int_{-1-iA}^{1+iA} g(z)e^{\rho h_q(z)} dz =0.
\]
We are then led to the estimate of $I^{(1)}$ and $I^{(-1)}$; we observe that 
\[
I^{(1)} = i \int_0^\infty e^{-i\rho}e^{- \rho t} (it)^{\gamma - 1 - i\alpha_N} (2-it)^{\gamma + i \alpha_N}.
\]
Therefore, for any $\gamma \ge 0$
\[
\lvert I^{(1)} \rvert \le 2^\gamma \int_0^\infty e^{-\rho t} t^{\gamma-1}\bigg ( 1 + \frac{t^2}4 \bigg )^{\frac \gamma2} dt .
\]
Since $(1+a^2)^b \le (1+a)^{2b} \le e^{2ab}$ for all $a,b \in \R^+$, when $0 \le \gamma < 2\rho$ the above inequalities implies 
\[
\lvert I^{(1)} \rvert \le C 2^\gamma \int_0^\infty e^{- (\rho - \frac \gamma2 ) t } t^{\gamma -1} dt = C 2^\gamma \bigg (\rho - \frac \gamma2 \bigg )^{-\gamma} \Gamma(\gamma).
\]
Similarly, when $1 \le \gamma < 2\rho -1$, 
\[
\lvert I^{(-1)} \rvert \le C 2^{\gamma -1} \int_0^\infty e^{- (\rho - \frac{\gamma-1}2)t} t^\gamma dt = C 2^{\gamma-1} \bigg ( \rho - \frac{\gamma -1}2 \bigg )^{-\gamma -1} \Gamma (\gamma +1).
\]
Moreover, by Stirling's formula, as $\gamma \gg 1$
\[
\Gamma (\gamma) \lesssim \gamma^{\gamma - \frac 12} e^{-\gamma}, \quad \Gamma (\gamma +1) \lesssim (\gamma +1)^{\gamma - \frac 12} (\gamma +1) e^{-\gamma}.
\]
Combining these estimates with \eqref{moduleigen} and  \eqref{est_prefactor} we get
the desired estimate. \\
Let now $\lvert k \rvert =1$ and $\rho >1$. As before, we divide the integral $I_{\gamma_1, \rho,N} = I^{(-1)} + I^{(1)}$. By integration by parts we get
\[
I^{(1)} =  \int_0^{+ \infty} \frac{(it)^{\gamma_1 -i\alpha_N}}{\gamma_1 - i \alpha_N} \frac{d}{dt} \Big ( e^{\rho(1+t)} (2-it)^{\gamma + i\alpha_N} \Big).
\]
This implies, recalling that $\lvert \gamma_1 - i \alpha_N \rvert \sim_{N \to 0} C_\nu N^{-1}$ and $\frac 12 \le \gamma _1 \le 1$,
\[
\lvert I^{(1)} \rvert \le C (\rho+1) 2^{\gamma_1} \int_0 ^\infty t^{\gamma_1} e^{-\rho t} \Big (1 + \frac{t^2}4 \Big )^{\frac {\gamma_1} 2} \le c (\rho +1) 2^{\gamma_1} \Big (\rho - \frac {\gamma_1}2 \Big )^{- \gamma_1 -1} \Gamma (\gamma_1 +1).
\]
We can estimate $I^{(-1)}$ more easily:
\[
\lvert I^{(-1)} \rvert \le c 2^{\gamma_1} \int_0^{+\infty} e^{-\rho t} \Big ( 1 + \frac{t^2}4 \Big )^{\frac{\gamma_1-1}2 } t^{\gamma_1} dt \le c 2^{\gamma_1}  \Big (\rho - \frac {\gamma_1}2 \Big )^{- \gamma_1 -1} \Gamma (\gamma_1 +1).
\]
Combing these estimates with \eqref{moduleigen} and \eqref{est_prefactor} we get \eqref{esteigformuladcrad1}. \\
In order to balance the prefactor, in the case $\gamma < \frac 12$ we need an improved estimate. 
We define
\[
I \coloneqq \int_0^\infty e^{-\rho t } t^{\gamma_1 -1 + i \alpha_N} (2+it)^{\gamma_1 - i\alpha_N}dt
\]
so that $\lvert \bar I^{(1)} \rvert = \lvert I \rvert $. We then estimate $I$; we change variable $y=\rho t$ 
\[
\begin{split}
I &=  \rho^{-2\gamma_1} \int_0 ^\infty e^{-y} y^{\gamma_1 -1 + i\alpha_N} (2\rho +iy)^{\gamma_1 - i\alpha_N}dy\\
& =  \rho^{-2\gamma_1-1} \int_0 ^\infty e^{-y}  (2\rho +iy)^{2\gamma_1 +1 } \Big ( \frac 1{\gamma_1 + i \alpha_N} \Big ( \frac y{2\rho + iy} \Big )^{\gamma+ i \alpha_N} \Big )' dy
\end{split}
\]
Then, integrating by parts we get
\[
\lvert I \rvert  \lesssim \frac{\rho^{-2\gamma_1 -1} }{\lvert \gamma_1 + i \alpha_N \rvert } \int_0^ \infty e^{-y} y^{\gamma_1} \lvert 2\rho + i y \rvert^{\gamma_1} \, \lvert -(2\rho + iy)  +i(2\gamma_1 +1) \rvert e^{-\alpha_N Arg(2\rho +iy)}dy.
\]
We recall $\alpha_N < 0$ and $\rho > 1$ so that 
\[
\begin{split}
\lvert I \rvert & \lesssim \frac{\rho^{-2\gamma_1 -1} }{\lvert \gamma_1 + i \alpha_N \rvert } \int_0^ \infty e^{-y} y^{\gamma_1} \big ((2\rho)^{\gamma_1 +1} + y^{\gamma_1 +1} \big ) dy \\
& \lesssim N (\rho^{-\gamma_1} \Gamma (\gamma_1 +1) + \rho^{-2\gamma_1 -1 } \Gamma (2\gamma_1 +2) ) \lesssim N.
\end{split}
\]
With similar computations we get that the same bound holds also for $I^{(-1)}$. Then, combing these estimates with \eqref{moduleigen} and \eqref{est_prefactor_2} we get \eqref{esteigformuladcrad1}. 

\textbf{Sector $\frac{|k|}2 < \rho \le \frac {(\gamma +1)^2}2$, $|k| \ge 2$:} In this sector, we can adapt the estimates provided in \cite{cacserzha} for the generalized eigenfunctions related to the massless Dirac-Coulomb equation. Indeed, as remarked, the phase $h_q(t)$ in integral \eqref{integral_phase_ampl} is the same, and thus the curves in the steepest descent methods are the same. 
Moreover, in the half-plane $\{ z \in \C \colon \Im(z) <0 \}$ we have that the term $ Arg \big (\frac{1-z}{1+z} \big ) \in [0, \pi]$. Therefore, 
\[
\lvert g(z) \rvert = \lvert 1-z \rvert e^{\alpha_N Arg \big ( \frac{1-z}{1+z} \big)} \le \lvert 1-z \rvert
\]
which implies that the function $g$ is uniformly bounded, and this is enough to conclude. 
\end{proof}

From Theorem \ref{thm_formuladc} we can deduce the following

\begin{corollary}\label{finalcorDC}
    Let $\nu \in [-1, 0)$ and let $\Phi_k^{DC}(N,\rho)$ as in \Cref{thm_formuladc}. Then, for any $q \in [2, + \infty]$  and $\lvert k \rvert \ge 2$ there exists a constant $C >0$, independent on $k$ such that the following estimates hold
    \[
\big \lVert \Phi_k^{DC}  \big \rVert_{L^q ([R,2R])} \le C \times 
\begin{cases}  R^{\gamma + \frac1q - 1}, &\quad \text{ if } R \le 1, \\ 
R^{\beta (q)}, &\quad \text{ if } R \ge 1,
\end{cases}
\]
and
\[
\big \lVert (\Phi_k^{DC})' \big \rVert_{L^q ([R,2R])} \le C \times \begin{cases}
R^{\gamma + \frac 1q - 2}, &\quad R \le 1, \\
R^{\beta(q)}, &\quad R \ge 1,
\end{cases}
\]
where 
\[
\beta(q) = \begin{cases}
\frac1q - 1&\quad \text{ if } q \in [2,4),\\
\frac1q - \frac56 &\quad \text{ if } q \in [4, + \infty].
\end{cases}
\]
Moreover, if $\lvert k \rvert =1$ 
\[
\big \lVert \Phi_{\pm 1}^{DC}  \big \rVert_{L^q ([R,2R])} \le C \times 
\begin{cases}  R^{\gamma_1 + \frac1q - 1}, &\quad \text{ if } R \le 1, \\ 
R^{\frac 1q -1}, &\quad \text{ if } R \ge 1
\end{cases}
\]
and
\[
\big \lVert (\Phi_{\pm1}^{DC})' \big \rVert_{L^q ([R,2R])} \le C \times \begin{cases}
R^{\gamma_1 + \frac 1q - 2}, &\quad R \le 1, \\
R^{\frac 1q -1}, &\quad R \ge 1.
\end{cases}
\]
\end{corollary}
\begin{proof}
The proof is the same of the one of Lemma \eqref{lembessel}; we only need to make use of the estimates provided in Theorem \ref{thm_formuladc} which replace \eqref{estbes}. 
\end{proof}

\begin{remark}\label{criticalregion}

Let us clarify why the restriction onto the positive side of the spectrum is a necessary assumption for Theorem \ref{DCtheorem}. As it is well known, for negative energies $E<-m$ an explicit representation of the generalized eigenfunctions can be deduced by the positive one by making use of the charge conjugation operator.
In other words, the generalized eigenfunctions for negative energies can be obtained from the corresponding positive ones by implementing the following substitutions in \eqref{psiDC}
\begin{equation}\label{chargeconj}
    E\rightarrow -E, \qquad \nu\rightarrow -\nu, \qquad k\rightarrow -k,\qquad {\bf p} \rightarrow -{\bf p}
\end{equation}
and by ``exchanging the roles of $F$ and $G$''. Notice then that the term $e^{\frac\pi2 \alpha_E}$ tends to zero if $E\rightarrow +m$, while it blows up if $E\rightarrow -m$. Then, in the repulsive case, in the negative side of the spectrum, and close to the energy levels of the eigenfunctions, a (structural) unbounded term appears, which provides an unavoidable obstacle in view of proving the necessary pointwise estimates on the generalized eigenfunctions and thus the dispersive estimates of Theorem \ref{DCtheorem}. This fact does not come too much as a surprise, as indeed the point spectrum of the Dirac-Coulomb operator in the repulsive case in the gap $(-m,m)$ accumulates towards $-m$ (see formula \eqref{discspec}), and this provides an obstacle for the dispersion for energies $\sim m$.
\end{remark}

    \end{document}